\documentclass[10pt,oneside,shortlabels, reqno]{amsart}
\usepackage{color}
\usepackage{enumitem}
\usepackage{tikz}
\usepackage{tikz-cd}
\usepackage{amsbsy}
\usepackage{amstext}
\usepackage{amsthm}
\usepackage{amssymb}
\usepackage[unicode=true,pdfusetitle,
 bookmarks=true,bookmarksnumbered=false,bookmarksopen=false,
 breaklinks=false,pdfborder={0 0 0},pdfborderstyle={},backref=page,colorlinks=true]
 {hyperref}
\hypersetup{
 linkcolor=blue}

\makeatletter
\newcommand{\sideremark}[1]{\ifvmode\leavevmode\fi\vadjust{\vbox to0pt{\vss 
      \hbox to 0pt{\hskip\hsize\hskip1em           
 \vbox{\hsize2cm\tiny\raggedright\pretolerance10000
 \noindent #1\hfill}\hss}\vbox to8pt{\vfil}\vss}}}%

\usepackage{mathrsfs}

\usepackage[all]{xy}
\newcommand{\xyL}[1]{%
\xydef@\xymatrixrowsep@{#1}
} 
\newcommand{\xyC}[1]{%
\xydef@\xymatrixcolsep@{#1}
} 

\newcommand\thmsname{Theorem}
 \newcommand\nm@thmtype{thm}
 \theoremstyle{plain}
 
 \newenvironment{namedthm}[1]{
   \renewcommand\thmsname{#1}\renewcommand\nm@thmtype{namedtheorem}
   \begin{\nm@thmtype}
}
   {\end{\nm@thmtype}
}
\theoremstyle{plain}
\newtheorem{thm}{\protect\theoremname}[section]
\theoremstyle{definition}
\newtheorem{defn}[thm]{\protect\definitionname}
\theoremstyle{remark}
\newtheorem{rem}[thm]{\protect\remarkname}
\theoremstyle{definition}
\newtheorem*{example*}{\protect\examplename}
\theoremstyle{remark}
\newtheorem{notation}[thm]{\protect\notationname}
\theoremstyle{plain}
\newtheorem{prop}[thm]{\protect\propositionname}
\theoremstyle{plain}
\newtheorem{lem}[thm]{\protect\lemmaname}

\numberwithin{equation}{section}

\makeatother

\providecommand{\definitionname}{Definition}
\providecommand{\remarkname}{Remark}
\providecommand{\examplename}{Example}
\providecommand{\lemmaname}{Lemma}
\providecommand{\notationname}{Notation}
\providecommand{\propositionname}{Proposition}
\providecommand{\theoremname}{Theorem}

\begin{document}
\global\long\def\acc#1#2{\accentset#1#2}%

\title{LINEARIZATION OF COMPLEX HYPERBOLIC DULAC GERMS}
\author{D. PERAN$^1$, M. RESMAN$^2$, J.P. ROLIN$^3$, T. SERVI$^4$}
\thanks{This research of D. Peran and M. Resman is partially supported by the Croatian Science Foundation (HRZZ) grant UIP-2017-05-1020. The research of M. Resman is also partially supported
	by the Croatian Science Foundation (HRZZ) grant PZS-2019-02-3055 from Research Cooperability funded by the European Social Fund. The research of all four authors is partially supported
	by the Hubert-Curien `Cogito’ grant 2021/2022 \emph{Fractal and transserial approach to differential
	equations}.}

\begin{abstract}
We prove that a hyperbolic Dulac germ with complex coefficients in its expansion is linearizable on a standard
quadratic domain and that the linearizing coordinate is again a complex Dulac
germ. The proof uses results about normal forms of hyperbolic transseries
in \cite{prrs:normal_forms_hyperbolic_logarithmic_transseries}.
\end{abstract}

\maketitle
\tableofcontents{}

\section{Introduction}

Given the germ $f$ of a real or complex function in one variable
at the fixed point $0$, the \emph{linearization} of $f$ consists
in finding a number $\lambda$ and a change of coordinates $\varphi$
which solves \emph{Schr{\"o}der's equation $\varphi\left(f\left(z\right)\right)=\lambda\varphi\left(z\right)$}
\cite{schroder:ueber_iterirte_functionen}. In the differentiable
case, $\lambda$ is the \emph{multiplier} $f'\left(0\right)$.\emph{
}A goal of the linearization is to embed $f$ in a flow, which allows
to define the so-called ``fractional iterates'' of $f$ by the formula
$f^{\left[t\right]}\left(z\right):=\varphi^{-1}\left(\lambda^{t}\varphi\left(z\right)\right)$
for $t\in\mathbb{R}$ (or $t\in\mathbb{C}$). Schr{\"o}der's equation has been solved by
Koenigs in the case of a holomorphic germ $f$ at $0\in\mathbb{C}$
with a \emph{hyperbolic attractive} fixed point (that is, when $0<\left|\lambda\right|<1$)
\cite{koenigs:recherches_integrales_equations_fonctionnelles,carleson-gamelin:complex-dynamics,milnor:dynamics_one_comple_variable}.
Koenigs' method consists in proving that the so-called \emph{Koenigs
sequence} $\left(\varphi_{n}\right)$, defined by $\varphi_{n}=\dfrac{1}{\lambda^{n}}f^{\circ n}$,
converges uniformly to a holomorphic solution $\varphi$ called the
\emph{Koenigs linearizing coordinate}. Moreover, $\varphi$ is a biholomorphism
which is tangent to the identity, that is $\varphi\left(z\right)=z+o(z)$.\emph{
}\\

The problem of the convergence of Koenigs' sequence for more general
germs has been tackled by various authors. Kneser proved the convergence
of the Koenigs sequence for a hyperbolic attracting real germ $f$ such
that $f\left(x\right)=\lambda x+O\left(\left|x\right|^{1+\delta}\right)$ for $x\to0$,
for some $\delta>0$ \cite{kneser:reelle_analytische_losungen}. However,
this result does not help much for the purposes of iteration since,
under Kneser's hypotheses, the limit $\varphi$ of Koenigs' sequence
may not admit a compositional inverse (so that $\varphi$ is not,
strictly speaking, a change of coordinates). To overcome this problem,
Szekeres imposed stronger conditions on $f$ \cite{szekeres:regular_iteration_real_complex}.
He proved that if $f$ is continuous and has a differentiable representative
on an interval $\left(0,d\right)$ which is strictly increasing, with
$0<f\left(x\right)<x$ on $\left(0,d\right)$, and if $f'\left(x\right)=\lambda+O\left(x^{\delta}\right)$ for $x\to0$,
for some $0<\lambda<1$ and $\delta>0$, then its Koenigs sequence
converges on $\left(0,d\right)$ to a differentiable and strictly
increasing solution. A simpler proof for real germs of class $\mathcal{C}^{r}$,
$r\ge2$, has been provided by Sternberg \cite{sternberg:local_cn_transformations_real_line,navas:groups_circle_diffeomorphisms}.\\

More recently, the convergence of Koenigs' sequences for maps admitting
an asymptotic behavior in the scale of iterated logarithms has been
considered \cite{dewsnap-fischer:convergence-koenigs-sequences}.
More precisely, if $f$ is an interval map of class $\mathcal{C}^{1}$
defined in a neighborhood of the fixed point $0$ with $f'\left(0\right)=\lambda$,
$0<\lambda<1$, such that
\begin{equation}\label{eq:type}
f\left(x\right)=\lambda x+O\left(\frac{x}{y\log\left(y\right)\cdots\log^{\circ (p-1)}\left(y\right)\left(\log^{\circ p}\left(y\right)\right)^{1+\varepsilon}}\right)
\end{equation}
for $x\to0$, for some $\varepsilon>0$ and a nonnegative integer $p$, where $y:=-\log\left(\left|x\right|\right)$,
then the Koenigs sequence of $f$ converges uniformly on a neighborhood
of $0$ to a limit $\varphi$ such that $\varphi\left(0\right)=0$
and $\varphi'\left(0\right)=1$. The exponent $\varepsilon>0$ plays
an important role here, as shown by the following example from \cite{sternberg:local_cn_transformations_real_line}.
If $f$ is defined by $f\left(x\right)=x\left(\lambda-\dfrac{1}{\log\left(x\right)}\right)$
for $x\in(0,d]$, $d>0$, and $f\left(0\right):=0$, then its Koenigs
sequence diverges on $(0,d]$ (see also \cite{dewsnap-fischer:convergence-koenigs-sequences,navas:groups_circle_diffeomorphisms}).\\

The previous result on functions with a logarithmic asymptotic behavior
leads us naturally to a class of maps studied increasingly in the
last few years, namely, maps which admit a \emph{transserial asymptotic
expansion} at $0$. In a word, a \emph{transseries} is a generalized
series whose monomials involve the exponential and logarithm functions.\emph{
}Maps (or germs of maps) which admit at their fixed point a transseries
as their asymptotic expansion are studied by physicists nowadays (see
for example \cite{aniceto_ince_bacsar_schiappa:primer_resurgent_transseries_asymptotics}).
They also appear in dynamical systems and differential equations,
e.g. as first return maps of polycycles of polynomial vector fields,
studied in detail by {\'E}calle \cite{ecalle:dulac} and Il'yashenko \cite{ilyashenko:dulac}
in their respective proofs of Dulac's Conjecture. In this case, while
being generated by a continuous dynamical system (a polynomial vector
field), these maps are viewed as discrete dynamical systems on the
real line. It is hence relevant to study them from the point of view
of iteration theory. Their properties reveal the features of the generating system. In particular, the multiplicity of such a map in a parametric family is linked to the cyclicity of the polycycle \cite{roussarie:book}. In this spirit, the orbits of some of them have
been recently analyzed from the point of view of ``fractal analysis'' to read the cyclicity in bifurcations or formal normal forms \cite{zu-zu:poincare_map_in_fractal_analysis,mardesic_resman-zupanovic:multiplicity_fixed_points,resman-formal-parabolic}.\\

Among these first return maps, some are of particular importance.
In this paper, we consider \emph{Dulac maps}, which are called \emph{almost regular germs} by Il'yashenko \cite{ilyashenko:nondegenerate_translated}. The first return maps of nondegenerate polycycles of saddle type belong to the class of Dulac germs.
They are analytic on an open interval $\left(0,d\right)$, and their
(trans)asymptotic expansion at $0$ is a (possibly) infinite sum of
powers of the variable multiplied by real polynomials in the logarithm
of the variable (what Il'yashenko calls in \cite{ilyashenko:nondegenerate_translated} a \emph{Dulac series}). Moreover, a
remarkable achievement of Il'yashenko is the following \emph{quasianalyticity}
result: a Dulac map is equal to the identity if and only if its asymptotic
expansion is equal to the identity. In order to prove this, he showed (considering complexifications in $\mathbb C^2$ of planar saddles)
that Dulac maps can be analytically extended to sufficiently big complex domains
(called \emph{standard quadratic domains}) of the Riemann surface
of the logarithm. Then, due to the fact that these extensions decay
faster than exponentially at infinity on these domains, the conclusion
follows by a version of the maximum modulus principle, the Phragmen-Lindel{\"o}f
theorem \cite{ilyashenko:nondegenerate_translated}.

The Dulac maps play a crucial role in Il'yashenko's study of Dulac's Conjecture about non-accumulation of limit cycles on elementary polycycles,
and have been extensively studied since. In particular, we are interested
here in their normal forms and their embeddings in flows of vector
fields, that is, in their study from the point of view of iteration
theory.\\

This approach has been initiated in \cite{mrrz:fatou} and continued
in \cite{mardesic_resman:moduli_parabolic}. There, the authors consider
Dulac maps \emph{tangent to the identity}, that is, with asymptotic
expansions of the form $z+o\left(z\right)$, when $z\rightarrow0$.
It was proved that such a Dulac map can be conjugated on attracting
and repelling sectors for its local dynamics, via a (sectorially)
analytic change of coordinate called a \emph{Fatou coordinate}, to
the translation $t\mapsto t+1$. This sectorial Fatou coordinate admits
a transserial asymptotic expansion, which is, however, more complicated
than a Dulac series.\\

Our main result here is the linearization of \emph{complex hyperbolic Dulac
germs} (Theorem B). These are Dulac maps defined on standard quadratic domains whose
asymptotic Dulac series, with possibly complex coefficients, are of the form $\lambda z+o\left(z\right)$, for $\lambda\in\mathbb C$, 
$0<|\lambda|<1$, uniformly in the domain as $|z|\to0$. We prove that a complex Dulac map
admits on a standard quadratic domain a linearizing Koenigs coordinate,
which is itself a complex Dulac map tangent to the identity (Theorem B in Section \ref{sec:notation main result}). Although the Dulac germs appearing as first return maps around saddle type polycycles of planar vector fields have only real coefficients in the expansion, the interest of complex Dulac germs lies in the fact that they appear as corner maps of hyperbolic complex saddles in $\mathbb C^2$, see e.g. Section 7 in \cite{loray:pseudo_groupe_singularite}.\\

Two main tools are used in the proof. First, we prove a general result,
which is an extension of Koenigs' and Dewsnap-Fisher's results: if
$f(z)=\lambda z+o(z)$ when $|z|\to0$, and $0<|\lambda|<1$, is a hyperbolic analytic map on a convenient invariant subdomain of the Riemann surface of the logarithm, and if it
has on this domain an asymptotic behavior similar to the one considered
in \cite{dewsnap-fischer:convergence-koenigs-sequences}, then the
Koenigs sequence of $f$ converges on the same subdomain to an analytic
linearizing map $\varphi$ tangent to the identity (Theorem A in Section
\ref{sec:analytic linearization dewsnap}). Note that we do not request
in Theorem A any particular asymptotic behavior after the first term.
Secondly, using the results of \cite{prrs:normal_forms_hyperbolic_logarithmic_transseries}
on the linearization of hyperbolic logarithmic transseries, we prove
that, if $f$ is a hyperbolic Dulac map on a quadratic domain, then
its Koenigs coordinate is also a Dulac map (see Section \ref{sec:linearization Dulac maps}).
This last fact marks a difference with the parabolic case studied
in \cite{mrrz:fatou}, where it was proven that the normalizing coordinate
of a parabolic Dulac map is in general not a Dulac map, but belongs
to some bigger class of germs with logarithmic transserial asymptotic
expansions.\\

Our theorems are thus a generalization of the standard Koenigs linearization
result for hyperbolic analytic germs from $\mathrm{Diff}(\mathbb{C},0)$
to hyperbolic germs which are not necessarily analytic at a fixed
point, on their invariant subdomains of the Riemann surface of the logarithm.

\section{\label{sec:notation main result}Notation and main results}

\subsection{\label{subsec:definitions and notation}Definitions and Notation}

In what follows, $z$ denotes a complex infinitesimal variable.\\

We first give the definition of a \emph{complex Dulac series} and of a \emph{complex Dulac germ}. We are motivated by the definition of \emph{Dulac series} and of \emph{almost regular germs} given in \cite{ilyashenko:dulac}. However, while the definitions there were adapted to the Dulac problem for real vector fields, we work here with series with complex coefficients.

\medskip

In order to define an almost regular germ, we recall a few classical
definitions. We denote by $\widetilde{\mathbb{C}}:=\left\{ \left(r,\theta\right):r\in\mathbb{R}_{>0},\theta\in\mathbb{R}\right\} $
the \emph{Riemann surface of the logarithm}. By a classical abuse
of notation, we write its elements $\left(r,\theta\right)$ as $z=r\mathrm{e}^{\mathrm{i}\theta}$,
where $r=\left|z\right|>0$ and $\theta=\arg(z)\in\mathbb{R}$.
A \emph{(spiraling) neighborhood of the origin} in $\widetilde{\mathbb{C}}$
is a set of the form $\mathcal{V}=\left\{ r\mathrm{e}^{\mathrm{i}\theta}:0<r<h\left(\theta\right)\right\} $,
where $h:\mathbb{R}\rightarrow\left(0,+\infty\right)$ is a continuous
function. Two functions define the same \emph{germ
at the origin of} $\widetilde{\mathbb{C}}$ if they coincide on some
spiraling neighborhood of the origin in $\widetilde{\mathbb{C}}$.

We endow $\widetilde{\mathbb{C}}$ with a structure of one-dimensional analytic Riemann
manifold whose atlas consists of a single chart, called the \emph{logarithmic
chart}, 
\[
-\log\!:\widetilde{\mathbb{C}}\rightarrow\mathbb{C},\thinspace z=r\mathrm{e}^{\mathrm{i}\theta}\mapsto\zeta=-\log z=-\log r-\mathrm{i}\theta.
\]
Hence $z=\exp\left(-\zeta\right)=\mathrm{e}^{-\zeta}$. Notice that
$z\rightarrow0$ (in the sense that $\left|z\right|\rightarrow0$)
when $\zeta\rightarrow\infty$ on $\mathbb{C}$ (in the sense that
$\Re\left(\zeta\right)\rightarrow+\infty$). 

A function $\widetilde{f}\colon \widetilde{A}\subseteq\widetilde{\mathbb{C}}\rightarrow\widetilde{\mathbb{C}}$
is called \emph{analytic} if its representation $f:\zeta\mapsto-\log\left(\widetilde{f}\left(\mathrm{e}^{-\zeta}\right)\right)$
on the domain $A=-\log\left(\widetilde{A}\right)$ of the logarithmic chart $\mathbb{C}$
is analytic. In what follows, we often switch for convenience between
the presentation of a domain or a function in the $z$-chart and its presentation in the logarithmic chart $\zeta=-\log z$ (also called the $\zeta$-chart).

Finally, let $\mathbb{C}^{+}=\{\zeta\in\mathbb{C}:\ \Re(\zeta)>0\}$.\\

Following \cite[Section 24]{ilyashenko-yakovenko:lectures_analytic_differential_equations}, given $C>0,$ the \emph{standard quadratic domain} $\mathcal{R}_{C}\subset\mathbb{C}$
is the set defined in the logarithmic
chart $\mathbb{C}$ as
\begin{equation}
\kappa\big(\mathbb{C}^{+}\big),\text{ where }\kappa(\zeta)=\zeta+C(\zeta+1)^{\frac{1}{2}},\label{eq:sqd}
\end{equation}
(see Figure~\ref{fig:prva}). The set $\widetilde{\mathcal{R}}_{C}\subset\widetilde{\mathbb{C}}$
represented by $\mathcal{R}_{C}$ is a spiraling
neighborhood of the origin in $\widetilde{\mathbb{C}}$. 
\begin{defn}[The \emph{germ} of a standard quadratic domain $\mathcal R_C$]\label{def:gsqd}
For $C>0$, let $\mathcal R_C\subset\mathbb{C}$ be the standard quadratic domain given by \eqref{eq:sqd} in the $\zeta$-chart. We call the elements of the collection of sets 
$$\{(\mathcal{R}_{C})_{R}=\mathcal{R}_C\cap ([R,+\infty) \times \mathbb{R}) : R>0\},$$
the \emph{germs} of the standard quadratic domain $\mathcal R_C$. 
\end{defn}
\begin{rem}\label{rem:gdomain}
	Let $C>0$. For each germ $(\mathcal R_C)_R$ of the standard quadratic domain $\mathcal R_C$, there exists sufficiently big $C_0>R$ such that for every $C'>C_0$, standard quadratic domain $\mathcal R_{C'}$ is entirely contained in $(\mathcal R_C)_R$.  Indeed, for $C'>C_0$, where $C_0$ is sufficiently big, it can be seen that $\mathcal R_{C'}\subset \mathcal R_C$. Also, for every $\zeta\in\mathcal R_{C'}$, it holds that $\Re(\zeta)>C'>C_0>R$.
\end{rem}
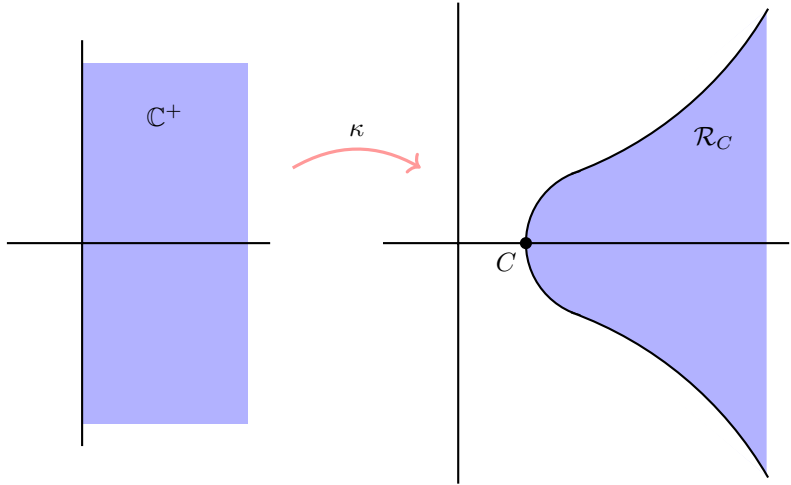
\begin{figure}[h!]
\centering
	\begin{tikzpicture}
		\filldraw[draw=none,fill=blue!30!white] (0,-2.4) rectangle (2.2,2.4);
		\draw (3.65,1.5) node{$\kappa $};
		\filldraw[draw=none,fill=blue!30!white] (9.1,-3.1) -- (6.6,-0.96) -- (6.6,0.96) -- (9.1,3.1) -- cycle;
		\filldraw[draw=black,thick,fill=blue!30!white] (6.9,0) circle (1cm);
		\filldraw[draw=none,thick,fill=blue!30!white,rounded corners] (6.62,-1.05) rectangle (8.5,1.05);
		\filldraw[draw=black,thick,fill=white] (6.5,0.92) arc (290:330:5cm);
		\filldraw[draw=black,thick,fill=white] (6.5,-0.92) arc (70:30:5cm);
		\draw[red!40!white,very thick,->] (2.8,1) to [in=150,out=30] (4.5,1);
		\draw (5.9,0) node[draw=black,thick,fill=black,circle,scale=0.4]{};
		\draw (5.9,0) node[below left]{$C$} (1.1,1.7) node{$\mathbb C^+$} (8.4,1.4) node{$\mathcal{R}_{C}$};
		\draw[black,thick] (-1,0) -- (2.5,0) (0,-2.7) -- (0,2.7) (4,0) -- (9.4,0) (5,-3.2) -- (5,3.2);
	\end{tikzpicture}
\caption{The image of standard quadratic domain $\mathcal{R}_{C}$, for some
$C>0$, in the $\zeta $-chart.}
\label{fig:prva}
\end{figure}

In analogy with \cite[Section 24]{ilyashenko-yakovenko:lectures_analytic_differential_equations}, a \emph{complex
Dulac series} is a transseries of the form
\begin{equation}
\widehat{f}=\alpha\zeta+\beta+\sum_{i=1}^{\infty}P_{i}(\zeta)\exp(-\alpha_i\zeta),
\quad\alpha\in\mathbb{R}_{>0},\thinspace\beta\in\mathbb{C},\thinspace P_i\in\mathbb{C}[\zeta],\label{eq:Dulac}
\end{equation}
where $(\alpha_{i})_{i\ge1}$ is a strictly
increasing sequence of positive real numbers belonging to a finitely generated
sub-semigroup of $\left(\mathbb{R}_{>0},+\right)$, such that $\left(\alpha_{i}\right)_{i}\to+\infty$.

The series $\widehat{f}$ is called \emph{parabolic} if $\alpha=1$ and $\beta=0$, and is called \emph{hyperbolic} if $\alpha=1$ and 
$\Re(\beta)\ne0$.\\

A \emph{complex Dulac germ} is
a holomorphic germ $f$ on a standard quadratic domain $\mathcal{R}_{C}$, 
which admits on $\mathcal{R}_{C}$ an asymptotic expansion 
given by a \emph{complex Dulac series} \eqref{eq:Dulac}, uniformly on $\mathcal{R}_{C}$
in the following sense: for every $\nu>0$, there exists $N_{\nu}\in\mathbb{N}$,
such that 
\begin{equation}
\big|f(\zeta)-\alpha\zeta-\beta-\sum_{i=1}^{N_{\nu}}P_{i}(\zeta)\exp(-\alpha_i\zeta)\big|=o(\exp(-\nu\zeta))),\label{eq:ekspa}
\end{equation}
as $\Re(\zeta)\rightarrow+\infty$ in $\mathcal{R}_{C}$.

\begin{rem}
A complex Dulac germ $f$ is represented in the $z$-chart by the germ $\widetilde{f}$ which admits as $z\rightarrow0$ an asymptotic expansion which is a \emph{logarithmic complex Dulac series}, that is a transseries 
\begin{equation}
\lambda z^\alpha+\sum_{i=1}^{+\infty}z^{\beta_i}Q_i(-\log z),\quad \alpha>0,\thinspace\lambda\in\widetilde{\mathbb{C}},\thinspace Q_i\in\mathbb{C}[X],
\label{eq:log-dulac-series}
\end{equation} 
where $(\beta_i)_{i\ge1}$ is a strictly increasing sequence of real numbers strictly bigger than $\alpha$, belonging to a finitely generated sub-semigroup of $(\mathbb{R}_{>0},+)$, and which tends to $+\infty$.
\end{rem}

If a complex Dulac series $\widehat f$ from \eqref{eq:Dulac} satisfies additionally $\beta\in\mathbb{R}_{\ge0}$ and $P_i\in\mathbb{R}[\zeta]$ for all $i\ge1$, then we call $\widehat f$ a \emph{real Dulac series}. If a  complex Dulac germ $f$ additionally satisfies that the image of $\{\zeta\in\mathcal R_C:\Im(\zeta)=0\}$ is again a subset of $\{\zeta\in\mathcal R_C:\ \Im(\zeta)=0\}$, then $f$ admits a real Dulac asymptotic expansion \eqref{eq:Dulac}, and we call $f$ a \emph{real Dulac germ} or an \emph{almost regular germ} \cite{ilyashenko:dulac}. 
\medskip

The \emph{complex} Dulac germs appear naturally as corner maps of hyperbolic complex saddles in $\mathbb C^2$. For example, consider the corner map of a complex saddle 
\begin{align}\label{eq:saddle}
\begin{cases}
z'&=z+O(2),\\
w'&=-\alpha_0 w+O(2),
\end{cases}
\end{align}
(where $O(2)$ are complex polynomials in the variables $z$ and $w$ of order at least $2$, and $\alpha_0\in\mathbb R_{>0}$),
realized between a pair of transversals $\{w=1\}\simeq \widetilde{\mathbb C}$ (horizontal) and $\{z=1\}\simeq {\widetilde{ \mathbb C}}$ (vertical) with canonical parametrizations as unit disks, or between any analytic reparametrizations of these transversals. By  Section 7 in \cite{loray:pseudo_groupe_singularite} this is a complex Dulac map. Note that the domain of definition of such corner maps is also a standard quadratic domain $\mathcal R_C$. Indeed, any complex saddle vector field \eqref{eq:saddle} is orbitally analytically equivalent to a normal form (see Section 22C in \cite{ilyashenko-yakovenko:lectures_analytic_differential_equations}):
\begin{align}\label{eq:normal-form-vect-field}
\begin{cases}
z'&=z,\\
w'&=w\left(-\alpha_0 +h(z,w)\right),
\end{cases}
\end{align}
where $h$ is a \emph{complex} analytic germ in two variables at $(0,0)$ and $h(z,w)=O(zw)$, as $z,\,w\to 0$. Computing a Dulac corner map of a complex saddle \eqref{eq:normal-form-vect-field}, exactly in the same way as it was done by Ilyashenko for complexified real saddles (see e.g. Section 3, Proof of Theorem 7.7 in \cite{roussarie:book}), we obtain that the domain of definition of the analytic corner map in the logarithmic chart has \emph{exponential growth}:
$$
\mathcal E_{C,M}:=\left\{\zeta\in\mathbb C:\,|\Im(\zeta)|\leq C \mathrm{e}^{M\Re(\zeta)}\right\},\,C,\,M>0,
$$
and it contains standard quadratic domains. 

Note that the saddle corner maps defined by foliation \eqref{eq:normal-form-vect-field} are not univalued, and that the procedure described in the proof of Theorem 7.7 in \cite{roussarie:book} provides just one possible determination. Indeed, the time being now a complex variable, one moves along the foliation described by \eqref{eq:normal-form-vect-field} from the point $(z,1)$ on the horizontal transversal to the point $(1,D(z))$ on the vertical transversal in complex time $-\log z=-\log|z|-\mathrm{i}\theta+2k\pi\mathrm{i},\ k\in\mathbb Z$. Depending on which determination of the logarithm one chooses (adding $2k\pi\mathrm{i},\ k\in\mathbb Z$, to the complex time), one obtains a different determination of the Dulac map $D(z)$. This corresponds to moving from the point $z$ along different paths of the complex foliation, that, when projected to the horizontal or the vertical separatrix, may include circling around the axes or not (resulting in composing $D$ with holonomies of the vertical and horizontal axes, which are non-ramified germs). For more details, see \cite{loray:pseudo_groupe_singularite}. However, no matter which determination we choose, $D(z)$ is a complex Dulac germ defined on a standard quadratic domain.
\smallskip

In \cite{ilyashenko:nondegenerate_translated} Il'yashenko proves
the following \emph{quasianalyticity} property of Dulac germs: if the Dulac
expansion of a real Dulac germ on a standard quadratic domain is just the identity,
then the germ itself is equal to the identity. The same property, based on Phragmen-Lindel\" of's maximum principle, can be proven similarly for complex Dulac germs, independently of the property of invariance of $\mathbb R_{\ge0}$ which is not necessarily satisfied for complex Dulac germs.


\subsection{The main result: Theorem B}

Consider an analytic germ at $+\infty$ in the $\zeta$-chart $\mathbb{C}$. 
We say that $f$ is: 

\begin{enumerate}
\item \emph{parabolic} if $f(\zeta)=\zeta+o(1)$ as $|\zeta|\to+\infty$, and if $f^{\circ q}\neq\mathrm{id}$
for all $q\in\mathbb{N}_{\geq 1}$, 
\item \emph{hyperbolic} if $f(\zeta)=\zeta+\beta+o(1)$, as $|\zeta|\to+\infty$, for some $\beta\in\mathbb{C}^+$ (we can always suppose that this is the case up to replacing $f$ by $f^{-1}$).
\end{enumerate}

For all our results, every statement on germs of maps means, as usual,
the similar statement for some representative of these germs. Our
version of Koenigs' linearization in the setting of hyperbolic complex Dulac
germs is the following:
\begin{namedthm}
{Theorem B}[Linearization of hyperbolic complex Dulac germs] Let $f(\zeta)=\zeta+\beta+o(1)$, $\beta\in\mathbb{C}^{+}$, be a hyperbolic
complex Dulac germ on a standard quadratic domain $\mathcal R_C$. Then there exists a unique
parabolic germ $\varphi$  
satisfying
\begin{equation}
\varphi\circ f=\varphi+\beta,\label{eq:linera}
\end{equation}
on $f$-invariant germs of $\mathcal R_C$.
Moreover, $\varphi$ is a complex parabolic Dulac germ $($possibly on a smaller standard quadratic subdomain $\mathcal R_{C'}\subset\mathcal R_C)$. Furthermore, if $f$ is a real Dulac germ, then $\varphi$ is also a real Dulac germ.
\end{namedthm}
Notice that the linearization equation \eqref{eq:linera}, written in the $\zeta$-chart, is an \emph{Abel-type equation} which says that $\varphi$ conjugates $f$ to the
translation by $\beta$. Written in the $z$-chart, it would become a \emph{Schr\"{o}der-type equation}
\begin{equation}
\widetilde{\varphi}\circ\widetilde{f}=\lambda\widetilde{\varphi},\label{eq:lineralog}
\end{equation}
where $\widetilde{\varphi}(z):=\exp(-\varphi(-\log{z}))$ and $\lambda=\exp(-\beta)\in\widetilde{\mathbb{C}}$.
\smallskip

The proof of Theorem~ B is in Section~\ref{sec:linearization Dulac maps}. In the statement of Theorem B, the linearization of a Dulac germ is done on an invariant germ of a \emph{standard quadratic} domain. The proof would also go through on every other invariant domain (e.g. constructed from various admissible domains from Subsection~\ref{subsec:admi}). Quadratic domains are chosen only for convenience, since they are the standard domains for Dulac germs, introduced in \cite{ilyashenko:nondegenerate_translated}. In the logarithmic chart, such domains are biholomorphic to $\mathbb C^+$, and hence the quasianalyticity property holds \cite{roussarie:book, ilyashenko:nondegenerate_translated}.  
\smallskip

In Section \ref{sec:analytic linearization dewsnap}, we prove a version (Theorem A in Subsection~\ref{subsec:thmA})
of this linearization result for less restrictive analytic maps defined
on invariant spiraling neighborhoods of the origin of $\widetilde{\mathbb{C}}$ (constructed from their various \emph{admissible} domains),
whose asymptotic behavior is bounded by a particular logarithmic term.
This statement, proven via the study of the convergence of the Koenigs
sequences, is a generalization to the complex setting of the original Koenigs result and of
the theorem of Dewsnap and Fisher in \cite{dewsnap-fischer:convergence-koenigs-sequences}.

The refinement about the Dulac nature of the linearizing coordinate
in the case of (complex) Dulac maps is proven in Section \ref{sec:linearization Dulac maps}.
It is based on our results on linearization of (formal) hyperbolic
transseries in \cite{prrs:normal_forms_hyperbolic_logarithmic_transseries}
and on the resolution of a particular homological equation.

\begin{rem} Note that the Dulac nature of the linearization $\varphi$ in Theorem~B is important: the quasianalyticity result of Ilyashenko \cite{ilyashenko:nondegenerate_translated} for real Dulac germs on standard quadratic domains (that is easily adaptable to complex Dulac germs) implies that the formal linearization $\widehat\varphi$ of the real Dulac expansion $\widehat f$ of a hyperbolic real Dulac germ $f$ obtained in \cite{prrs:normal_forms_hyperbolic_logarithmic_transseries} $($and repeated here in Lemma~\ref{lem:formal dulac linearization}$)$ \emph{uniquely determines} the $($unique$)$ analytic linearization $\varphi$ of $f$ on a standard quadratic domain. Therefore it is sufficient to work only with formal series.
\end{rem}

\section{\label{sec:analytic linearization dewsnap}Analytic linearization}

The purpose of this section is to show a general complex extension
of the results of \cite{dewsnap-fischer:convergence-koenigs-sequences}:
if a holomorphic map, defined on an appropriate invariant domain of the Riemann
surface of the logarithm, has an asymptotic behavior similar to the
one described in \cite{dewsnap-fischer:convergence-koenigs-sequences},
then it can be linearized by a parabolic change of coordinates obtained as the
uniform limit of its Koenigs sequence in the logarithmic chart. In particular,
we prove in Example $(3)$ below that the standard quadratic domains defined in Section \ref{sec:notation main result}
(on which the Dulac maps are usually considered) are particular types of such admissible (invariant) domains for the complex germs mentioned above. 
The general linearization result of this section will be used in the
proof of our main result (Theorem B, Section \ref{sec:linearization Dulac maps}),
which asserts the linearizability of hyperbolic Dulac maps by parabolic
Dulac changes of coordinates on standard quadratic domains. \\

In Subsection~\ref{subsec:admi}, we give the definition of an \emph{admissible domain}. These provide the main ingredient for producing invariant domains for hyperbolic complex holomorphic germs on $\widetilde{\mathbb C}$ with asymptotic behavior of type \eqref{eq:asy}, as stated in Proposition~\ref{prop:invariance of D}. 

In Subsection~\ref{subsec:ex}, we give several examples of such
domains. We show that standard quadratic domains and similar types of domains (bounded by curves of any growth $x^r,\ r>0,$ at infinity in the logarithmic chart) are invariant for all hyperbolic germs of type \eqref{eq:asy}, in particular, for complex hyperbolic Dulac germs. Finally, in Subsection~\ref{subsec:thmA} we state and prove our linearization
result on invariant domains for hyperbolic maps with complex multipliers
and an asymptotic behavior as in \cite{dewsnap-fischer:convergence-koenigs-sequences}
(Theorem A). The computations of this subsection are to some extent motivated by a normalization method for parabolic analytic germs in $(\mathbb C,0)$ described in \cite{Loray_series_divergentes}, and by Koenigs' linearization theorem for hyperbolic analytic germs in $(\mathbb C,0)$ described for example in \cite{carleson-gamelin:complex-dynamics}.\\

In this section, all the subdomains of the Riemann surface of the logarithm and all the maps defined on these domains are described in the logarithmic chart $\zeta=-\log z$.

\subsection{Admissible domains}\label{subsec:admi}

Recall that $\mathbb{C}^+:=\{\zeta\in\mathbb{C}:\Re(\zeta)>0\}$. Let $\beta \in \mathbb{C}^+$, $\varepsilon>0$ and $k\in\mathbb{N}$. Let 
\begin{align}\label{eq:rho}
M_{\varepsilon,k}(x) & :=\frac{1}{x\log x\cdots(\log^{\circ k} x)^{1+\varepsilon}},\\
\rho_{\beta , \varepsilon,k}^{\pm }(x) & :=\Re(\beta) \pm M_{\varepsilon,k}(x),\quad\text{for }x\in\left(\exp^{\circ k}\left(0\right),+\infty\right).\nonumber 
\end{align}
Note that $M_{\varepsilon,k}$ is a positive,
strictly decreasing map tending to $0$, as $x\to+\infty$. Therefore, $\rho_{\beta , \varepsilon,k}^{-}$ is a strictly increasing map and $\rho_{\beta , \varepsilon,k}^{+}$ is a strictly decreasing map, both tending to $\Re(\beta) $ at infinity.  Furthermore, it is known
that the series $\sum_{n\in\mathbb{N}}M_{\varepsilon,k}(x+ny)$ converges
for every $x,y>0$ (this last fact, which was used in \cite{dewsnap-fischer:convergence-koenigs-sequences},
will also be used in the proof of Theorem A). \\

In order to define admissible domains of type $(\beta , \varepsilon , k)$, we first define two functions $h_l$ and $h_u$, whose graphs bound the domain from ``below" and from ``above". We distinguish three cases: $\Im(\beta)>0$, $\Im(\beta) =0$ and $\Im(\beta)<0$:
\begin{enumerate}
\item[(i)] \textbf{Case} $\Im(\beta)>0$. Let $t>\exp^{\circ k}(0)$ such that $\rho_{\beta ,\varepsilon,k}^{-}(x)>0$ and $\Im(\beta)- M_{\varepsilon , k}(x)>0,\ x\in[t,+\infty)$. Let $h_{l},\ h_{u}:[t,+\infty)\to\mathbb R$ be any two functions satisfying: 
\begin{enumerate}[(1), font=\textup, nolistsep, leftmargin=0.6cm]
\item $h_{l}(x)<h_{u}(x)$, $x\in[t,+\infty)$;
\item $h_{l}$ is a decreasing map on $[t,+\infty)$, or $h_{l}$ is an increasing map with property:
\begin{align*}
h_{l}(x+\rho_{\beta ,\varepsilon,k}^{+}(x))-h_{l}(x) & \leq \Im(\beta) -M_{\varepsilon,k}(x),\ x\in[t,+\infty);
\end{align*}
\item $h_{u}$ is an increasing map with property:
\begin{align*}
h_{u}(x+\rho_{\beta ,\varepsilon,k}^{-}(x))-h_{u}(x) & \geq \Im(\beta) +M_{\varepsilon,k}(x),\ x\in [t,+\infty).
\end{align*}
\end{enumerate}

\item[(ii)] \textbf{Case} $\Im(\beta) =0$. Let $t>\exp^{\circ k}(0)$ such that $\rho_{\beta ,\varepsilon,k}^{-}(x)>0$, $x\in[t,+\infty)$. Let $h_{l},\,h_{u}:[t,+\infty)\to\mathbb R$ be any two functions satisfying:
\begin{enumerate}[(1), font=\textup, nolistsep, leftmargin=0.6cm]
	\item $h_{l}(x)<h_{u}(x)$, $x\in[t,+\infty)$;
	\item $h_{l}$ is a decreasing map with property:
	\begin{align*}
		h_{l}(x+\rho_{\beta ,\varepsilon,k}^{-}(x))-h_{l}(x) & \leq -M_{\varepsilon,k}(x),\ x\in[t,+\infty);
	\end{align*}
	\item $h_{u}$ is an increasing map with property:
	\begin{align*}
		h_{u}(x+\rho_{\beta ,\varepsilon,k}^{-}(x))-h_{u}(x) & \geq M_{\varepsilon,k}(x),\ x\in[t,+\infty).
	\end{align*}
\end{enumerate}
\item[(iii)] \textbf{Case} $\Im(\beta)<0$. Let $t>\exp^{\circ k}(0)$ such that $\rho_{\beta ,\varepsilon,k}^{-}(x)>0$ and $-\Im(\beta) - M_{\varepsilon , k}(x)>0,\ x\in[t,+\infty)$. Let $h_{l},\ h_{u}:[t,+\infty)\to\mathbb R$ be any two functions satisfying: 
\begin{enumerate}[(1), font=\textup, nolistsep, leftmargin=0.6cm]
\item $h_{l}(x)<h_{u}(x)$, $x\in[t,+\infty)$;
\item $h_{l}$ is a decreasing map on $[t,+\infty)$:
\begin{align*}
h_{l}(x+\rho_{\beta ,\varepsilon,k}^{-}(x))-h_{l}(x) & \leq \Im( \beta) -M_{\varepsilon,k}(x),\ x\in[t,+\infty);
\end{align*}
\item $h_{u}$ is an increasing map, or a decreasing map with property:
\begin{align*}
h_{u}(x+\rho_{\beta ,\varepsilon,k}^{+}(x))-h_{u}(x) & \geq \Im( \beta) +M_{\varepsilon,k}(x),\ x\in[t,+\infty).
\end{align*}
\end{enumerate}
\end{enumerate}
\smallskip

A map $h_{l}:[t,+\infty)\to\mathbb R$ with property $(2)$ is called a \emph{lower}
\emph{map of type} $(\beta ,\varepsilon,k)$. A map $h_{u}:[t,+\infty)\to\mathbb R$
with property $(3)$ is called an \emph{upper map of type $\left(\beta ,\varepsilon,k\right)$}.
A pair $(h_{l},h_{u})$ of maps $h_{l},\, h_{u}:[t,+\infty)\to\mathbb{R}$, satisfying conditions $(1)-(3)$ above,
is called a \emph{lower-upper pair of type} $(\beta ,\varepsilon,k)$. Notice that the opposite of an upper map of type $(\beta,\varepsilon,k)$ is a lower map of type $(\beta,\varepsilon,k)$.
\smallskip

Finally, let
\begin{equation}
D_{h_{l},h_{u}}:=\left\lbrace \zeta\in\mathbb{C}^{+}:\mathrm{Re}\,\zeta\geq t,\ h_{l}(\mathrm{Re}\,\zeta)<\mathrm{Im}\,\zeta<h_{u}(\mathrm{Re}\,\zeta)\right\rbrace .
\end{equation}

\begin{defn}[Admissible domain]
\label{def:admissible domain lambda real}Let $\beta \in \mathbb{C}^+$, $\varepsilon>0$ and $k\in\mathbb{N}$. A \emph{domain of type} $(\beta ,\varepsilon,k)$ (or \emph{$(\beta, \varepsilon ,k)$-domain}) is defined as a union of an arbitrary nonempty collection of subsets of the form $D_{h_{l},h_{u}}\subseteq\mathbb{C}$ defined above. Similarly, a subset $D\subseteq\mathbb{C}$ which contains a $(\beta, \varepsilon ,k)$-domain is called an \emph{admissible domain of type} $(\beta ,\varepsilon,k)$ (or \emph{$(\beta, \varepsilon ,k)$-admissible domain}).
\end{defn}

\begin{rem}\label{rem:union} It follows from Definition~\ref{def:admissible domain lambda real} that an arbitrary union of domains of type $(\beta,\varepsilon,k)$ is again a domain of type $(\beta,\varepsilon,k)$. 
\end{rem}

\subsection{Examples}\label{subsec:ex}\

In this subsection, we fix $\beta \in \mathbb{C}^+$, $\varepsilon>0$ and $k\in\mathbb{N}$. 
We give here several examples of upper (lower) maps of type $(\beta ,\varepsilon,k)$ and of $(\beta,\varepsilon,k)$-admissible domains.

In particular, the fact that a standard quadratic domain as in \eqref{eq:sqd} is $(\beta,\varepsilon,k)$-admissible is proven in 
Example $(3)$.
\smallskip

We first provide a general technical sufficient condition under
which a map $h$ is an upper or a lower map. This condition is then used
in the following Examples $(2)$ and $(4)$. 

\textbf{Upper map condition in case $\Im(\beta)\geq0$.} Let $t>\exp^{\circ k}(0)$ such that $\rho_{\beta,\varepsilon,k}^-(t)>0$. Note that $\rho_{\beta ,\varepsilon,k}^{-}\left(t\right)\leq\rho_{\beta ,\varepsilon,k}^{-}(x)$
, for $x\in[t,+\infty)$.
Let $h\in\mathcal{C}^{n}\left([t,+\infty)\right)$, for some $n\in\mathbb{N}_{\geq1}$, be an increasing map.
Suppose that there exists a positive number $0<\rho<\rho_{\beta,\varepsilon,k}^{-}\left(t\right)$
such that 
\begin{align}
\sum_{i=1}^{n}\frac{h^{(i)}(x)}{i!}\rho^{i} & \geq \Im(\beta) +M_{\varepsilon,k}(x),\text{ for all }x\geq t,\label{eq:Condition4}
\end{align}
and that $h^{(n)}:[t,+\infty)\to\mathbb{R}$ is increasing. Then $h$ is an upper map of type $(\beta ,\varepsilon,k)$.

This can be seen as follows. Since $h$ and $h^{(n)}$ are increasing, it follows from Taylor's theorem and then \eqref{eq:Condition4}
that 
\begin{align*}
h(x+\rho_{\beta ,\varepsilon,k}^{-}(x))-h(x)&\geq h(x+\rho)-h(x)\geq\sum_{i=1}^{n}\frac{h^{(i)}(x)}{i!}\rho^{i}\\
&\geq \Im(\beta) + M_{\varepsilon,k}(x),\ x\in[t,+\infty).
\end{align*}

\textbf{Upper map condition in case $\Im(\beta) < 0$.} Let $t>\exp^{\circ k}(0)$. Suppose that $h:[t,+\infty)\to\mathbb R$ is either an increasing map, or a decreasing map belonging to $\mathcal C^n([t,+\infty))$ for some $n\in\mathbb N_{\geq1}$, which satisfies, for some positive number $\rho>\rho_{\beta ,\varepsilon,k}^{+}(t)$:
\begin{align}
\sum_{i=1}^{n}\frac{h^{(i)}(x)}{i!}\rho^{i} & \geq \Im(\beta) +M_{\varepsilon,k}(x),\text{ for all }x\geq t,\label{eq:Condition9}
\end{align}
and that $h^{(n)}:[t,+\infty)\to\mathbb{R}$ is increasing. Then $h$ is an upper map of type $(\beta ,\varepsilon,k)$.

\textbf{Lower map condition in case} $\Im(\beta)>0$. Let $t>\exp^{\circ k}(0)$ such that $\Im( \beta) - M_{\varepsilon , k}(t)>0$. Then, $\Im(\beta) - M_{\varepsilon , k}(x)>0,\ x\in[t,+\infty)$. Similarly as above, if $h$ is either decreasing on $[t,+\infty)$, or increasing belonging to $\mathcal C^n([t,+\infty))$ for some $n\in\mathbb N_{\geq1}$, and satisfying the property 
\begin{align*}
\sum_{i=1}^{n}\frac{h^{(i)}(x)}{i!}\rho^{i} & \leq \Im(\beta) - M_{\varepsilon,k}(x),\text{ for all }x\geq t,
\end{align*}
for some $\rho>\rho_{\beta ,\varepsilon,k}^{+}(t)$, and with $h^{(n)}$ decreasing on $[t,+\infty)$,
it follows that $h$ is a lower map of type $(\beta ,\varepsilon,k)$. 

\textbf{Lower map condition in case} $\Im(\beta) \leq 0$. Let $t>\exp^{\circ k}(0)$ such that $\rho_{\beta,\varepsilon,k}^-(t)>0$. If $h$ is a decreasing map belonging to $\mathcal C^n([t,+\infty))$ for some $n\in\mathbb N_{\geq1}$, and satisfying the property 
\begin{align}\label{eq:cond5}
	\sum_{i=1}^{n}\frac{h^{(i)}(x)}{i!}\rho^{i} & \leq \Im(\beta)-M_{\varepsilon,k}(x),\ \text{ for all }x\geq t,
\end{align}
for some $0<\rho<\rho_{\beta ,\varepsilon,k}^{-}(t)$, and if $h^{(n)}$ is decreasing,
it follows that $h$ is a lower map of type $(\beta ,\varepsilon,k)$.

\begin{example*}[1] (\emph{Sufficient condition for upper/lower maps})

Here, $\Im(\beta)\geq 0$.	Let $t>\exp^{\circ k}(0)$ be such that $\rho_{\beta ,\varepsilon,k}^{-}(t)>0$. Let $h:[t,+\infty)\to\mathbb{R}$ be an increasing map of class
	$\mathcal{C}^{1}$, such that:

1.  $h':[t,+\infty)\to\mathbb{R}$ tends to $\lambda '> \frac{\Im(\beta) }{\Re(\beta) }, $ as $x\to + \infty $. Then, since $M_{\varepsilon,k}(x)\to 0$ and $\rho_{\beta,\varepsilon,k}^{-} \to \Re(\beta)$ as $x\to+\infty$, there exists $t'\geq t$ sufficiently large such that
	\begin{equation}\label{eq:hprime_frac}
	h'(x)\geq \frac{\Im(\beta) +M_{\varepsilon ,k}(t')}{\rho _{\beta , \varepsilon ,k}^{-}(t')} , 
	\end{equation}
	for every $x\geq t'$. 
	Since $M_{\varepsilon ,k}$ is decreasing and $\rho_{\beta,\varepsilon,k}^{-}$ is increasing, for every $v\in(0,1)$ and $x>t'$,
	$$h'(x+v\rho_{\beta,\varepsilon,k}^{-}(x))\cdot \rho_{\beta,\varepsilon,k}^{-}(x)\geq \Im(\beta) +M_{\varepsilon,k}(x).$$ 
	Hence, by the Mean Value Theorem, the restriction $h|_{\left[t',+\infty\right\rangle }$ is an upper map of type $(\beta , \varepsilon ,k)$.
	\smallskip
	
	2. $h':[t,+\infty) \to \mathbb{R}$ tends to $+\infty $, as $x\to +\infty $. Then we choose $t'\geq t$ sufficiently large such that  \eqref{eq:hprime_frac} 
	holds. The restriction $h|_{\left[t',+\infty\right\rangle }$
	is therefore an upper map of type $(\beta ,\varepsilon,k)$.
\smallskip 
	
In the case $\Im(\beta)<0$, any increasing $h$ on $[t,+\infty)$ is an upper map of type $(\beta,\varepsilon,k)$. 

\smallskip

Analogously, a similar sufficient condition can be deduced for lower maps of type $(\beta,\varepsilon,k)$.
\end{example*}

\begin{example*}[2] (\emph{Maps of type $h(x)\sim x^r,\ r>0$}) 

1. \emph{Case $r>1$.} Let $h:[t,+\infty)\to \mathbb{R}$, $t>0$, be an increasing map of class $\mathcal{C}^{1}$ such that\footnote {We write $f\sim g$, $x\to\infty$, if $\lim_{x\to\infty}\frac{f(x)}{g(x)}=1.$}  \begin{equation}
\label{eq:hii}h(x)\sim ax^r,\ h'(x)\sim ar x^{r-1},\ a>0,\ r>1, \ x\to+\infty.
\end{equation} 
By Example (1.2), there exists $t>0$ big enough such that $h$ is an upper map of type $(\beta , \varepsilon ,k)$. 

Let $h:[t,+\infty)\to \mathbb{R}$, $t>0$, be a decreasing $\mathcal{C}^{1}$ map , such that \begin{equation}
\label{eq:hiii}h(x)\sim-ax^r,\ h'(x)\sim -ar x^{r-1},\ a>0,\ r>1,\ x\to+\infty.
\end{equation}
By a similar argument, there exists $t>0$ big enough such that $h$ is a lower map of type $(\beta , \varepsilon ,k)$.

2. \emph{Case $r=1$.} Let $h:[t,+\infty)\to \mathbb{R}$, $t>0$, be an increasing $\mathcal{C}^{1}$ map such that 
$$h(x)\sim ax,\ h'(x)\sim a,\  a>0,\ x\to +\infty.$$ Then, by Example (1.1), if $\Im(\beta)<0$, or if $\Im(\beta)\geq 0$ and $\frac{\Im(\beta)}{\Re(\beta)}<a$, there exists $t>0$ big enough such that $h$ is an upper map of type $(\beta , \varepsilon ,k)$. Let $h:[t,+\infty)\to \mathbb{R}$, $t>0$, be a decreasing $\mathcal{C}^{1}$ map satisfying  $h(x)\sim -ax$, $h'(x)\sim -a$, $a>0$, $x\to +\infty$. Similarly, if $\Im(\beta) >0$, or if $\Im(\beta) \leq 0$ and $\frac{\Im(\beta) }{\Re(\beta)}>-a$, there exists big enough $t>0$ such that $h$ is a lower map of type $(\beta,\varepsilon,k)$.

3. \emph{Case $0<r<1$.} The sufficient condition of Example (1) is not satisfied for $0<r<1$. However, suppose that $h$ is increasing (\emph{resp.} decreasing) of class $\mathcal{C}^{2}$, satisfying \eqref{eq:hii} (\emph{resp.} \eqref{eq:hiii}) with $r<1$, with the additional property that $h''(x)\sim ar(r-1)x^{r-2}$ (\emph{resp.} $h''(x)\sim -ar(r-1)x^{r-2}$) and $h''$ is increasing (\emph{resp.} decreasing). It follows from conditions \eqref{eq:Condition4} and \eqref{eq:cond5}, in a similar way as in Example (4) below, that if $\Im(\beta)=0$, then $h$ is an upper (\emph{resp.} lower) map of type $(\beta,\varepsilon,k)$, $\varepsilon>0$, $k\in\mathbb N$. 
\end{example*}

\begin{example*}[3] (\emph{Standard quadratic domains})

For $C>0$, let $\mathcal{R}_{C}\subseteq\mathbb{C}$
	be the standard quadratic domain defined in \eqref{eq:sqd}. The upper
	half of the boundary of $\mathcal{R}_{C}$ is described by a smooth
	function which satisfies sufficient conditions of Example $(1.2)$ (and the lower half satisfies
	the symmetric statement), hence such a domain is admissible. 
	
	Indeed,
	a direct computation shows that the boundary 
	\[
	\partial\left(\mathcal{R}_{C}\cap\left\lbrace \zeta\in\mathbb{C}:\mathrm{Im}\,\zeta\geq0\right\rbrace \right)
	\]
	can be parameterized by:
\begin{align*}
	r\to x(r)+\mathrm{i}\cdot y(r)=\ C&\sqrt[4]{r^{2}+1}\cos\left(\frac{1}{2}\mathrm{arctg}\,r\right)+\\
	&+\mathrm{i}\cdot\left(r+C\sqrt[4]{r^{2}+1}\sin\left(\frac{1}{2}\mathrm{arctg}\,r\right)\right),\ r\in\left[0,+\infty\right).
	\end{align*}
Note that $y:[0,+\infty)\rightarrow\mathbb{R}$
	is strictly increasing. Let $t>0$ be such that $x\left(t\right)>\exp^{\circ k}\left(0\right)$
	and $x$ is strictly increasing on $[t,+\infty)$. Therefore, $h_{u}:=y\circ x^{-1}$
	is strictly increasing on $[x(t),+\infty)$. By direct computation,
	if can be shown that the derivative of $h_{u}$ on $[x(t),+\infty)$ tends to $+\infty $, as $x\to + \infty $. Therefore, by Example (1.2),  there exists $x'>0$ such that the restriction $h_{u}|_{\left[ x',+\infty \right\rangle }$ is an upper map
	of type $\left(\beta ,\varepsilon,k\right)$. A similar argument can be repeated to show that an appropriate restriction of the lower boundary of $\mathcal R_C$ is the graph of a lower map of type $\left(\beta ,\varepsilon,k\right)$.
\end{example*}

\begin{example*}[4] (\emph{Logarithmic upper/lower maps})

\emph{Logarithmic upper maps.} Let $\beta \in \mathbb{C}^{+}$ such that $\Im (\beta )=0$. Let $h:[t,+\infty)\to\mathbb{R}$, $h(x):=(\log x)^{\delta }$, $\delta \in \mathbb{R}_{>0}$, $t\in \mathbb{R}_{>1}$. Note that:
\begin{align*}
	h'(x) & = \frac{\delta \, (\log x)^{\delta }}{x\log x} , \\
	h''(x) & = \frac{\delta \, (\log x)^{\delta }}{x\log x} \cdot \Big( \frac{\delta -1}{x\log x} - \frac{1}{x} \Big) , \quad x\geq t.
\end{align*}
For every $0<\rho < \rho _{\beta ,\varepsilon ,k}^{-}$, there exists $t>\exp ^{\circ k}(0)$ big enough such that 
\begin{align*}
	& h'(x)\rho +\frac{1}{2}h''(x)\rho ^{2}=\frac{1}{x\log x} \cdot \rho \, \delta (\log x)^{\delta }\cdot \Big( 1+\frac{\rho \, (\delta -1)}{2x\log x} - \frac{\rho }{2x} \Big) \\
	& \geq \frac{1}{x\log x}\cdot\frac{1}{\log^{\circ 2} x\cdots(\log^{\circ k} x)^{1+\varepsilon}}=M_{\varepsilon,k}(x),
\end{align*}
for $x\geq t$. It can be proven that we can take $t>\exp ^{\circ k}(0)$ big enough such that $h'''(x)>0$, for each $x\geq t$. This implies that restriction of $h''$ on $\left[ t,+\infty \right) $ is an increasing map. Therefore, since $h$ is increasing, it follows from sufficient upper map condition \eqref{eq:Condition4} that there exists $t>\exp ^{\circ k}(0)$ such that the restriction $h|_{\left[ t,+\infty \right) }$ is an upper map of type $(\beta ,\varepsilon,k)$. \\

\emph{Logarithmic lower maps.} Let $\beta \in \mathbb{C}^{+}$ such that $\Im (\beta )=0$ and let $h$ be as defined above. It follows that $g:\left[ t,+\infty \right) \to \mathbb{R}$, defined by $g(x):=-h(x)$, $x\in \left[ t,+\infty \right) $, is a lower map of type $(\beta ,\varepsilon,k)$.
\end{example*}
\smallskip

\subsection{Linearization theorem for holomorphic maps on spiraling domains: Theorem A}\label{subsec:thmA}\

\begin{notation}
	For every 
set $D\subseteq\mathbb C$ and $R>0$, let
	$$D_{R}:=D\cap\left([R,+\infty)\times\mathbb{R}\right)\subseteq\mathbb C.$$ Note that, if $D\subseteq \mathbb C$ is a (admissible) domain of type $(\beta , \varepsilon ,k)$, then $D_{R}$ is a (admissible) domain of type $(\beta , \varepsilon ,k)$. 

	For every admissible domain $D$ of type $(\beta ,\varepsilon,k)$,
	denote by $\overline{D}$  its \emph{maximal} subdomain of type $(\beta,\varepsilon,k)$. That is, by Remark~\ref{rem:union}, $\overline{D}$ is defined as the union of all subdomains of type $(\beta,\varepsilon,k)$ of $D$. 
	
	Denote by $D^f$
	the maximal $f$-invariant subdomain of an admissible domain $D\subseteq\mathbb{C}$ (i.e. the union of all $f$-invariant subdomains of $D$). 
	
	Finally, for an admissible domain $D$, define $$D^f_{R}:=(D^f)_{R},\ \overline{D}_{R}:=(\overline{D})_{R},\ R>0.$$	
\end{notation}

We extend the above definition to the case $D=\mathbb{C}^+$ by stipulating that $\overline{D}_R=\{\zeta\in\mathbb C^+: \Re(\zeta)>R\}$. For an admissible domain $D\subseteq \mathbb C$, the maximal $f$-invariant subdomain $D^f\subseteq D$ can be, in general, empty. The following proposition gives sufficient conditions such that $D^f$ and $D_R^f$, for every $R>0$, are non-empty.

\begin{prop}
	\label{prop:invariance of D} Let $\beta\in\mathbb C^+$, $\varepsilon>0$ and $k\in\mathbb{N}$. 
	Let $D\subseteq \mathbb C^+$ be an admissible domain of
	type $(\beta,\varepsilon,k)$. Let $f:D_C\to\mathbb C$, $C>\exp^{\circ k}(0)$,
	be an analytic map, such that
	\begin{equation}\label{eq:asy}
	f(\zeta)=\zeta+\beta +o(\zeta^{-1}\boldsymbol{L}_{1}^{-1}\cdots\boldsymbol{L}_{k}^{-\left(1+\varepsilon\right)}),\ 
	\text { as }\mathrm{\Re}(\zeta)\to +\infty \text{ uniformly on $D_C$}.
	\end{equation}
	Here, \[
	\boldsymbol{L}_{1}:=\log\left(\zeta\right),\ldots,\boldsymbol{L}_{k}:=\log\left(\boldsymbol{L}_{k-1}\right),
	\]
	where $\log$ represents the principal branch of the logarithm\footnote{Note here that, for $C>\mathrm{exp}^{\circ k}(0)$, the iterated logarithms $\boldsymbol{L}_1,\ldots,\boldsymbol L_k$ are well-defined on $D_C$ (using only the principal branch of the logarithm), since $\Re(\zeta)>\mathrm{exp}^{\circ k}(0)$.}.
	Then, for every $R>C$
	sufficiently large, the domain $\overline{D}_{R}$ is $f$-invariant.
	In particular, $\overline{D}_{R}\subseteq{D}_{R}^f$ and $D_R^f\neq \emptyset$. 
\end{prop}

\begin{proof}
	By asymptotics \eqref {eq:asy},
	\begin{align}\label{eq:asylog}
		{\displaystyle \lim_{\Re(\zeta)\to+\infty}\frac{f(\zeta)-\left(\zeta+\beta \right)}{\zeta^{-1}\boldsymbol{L}_{1}^{-1}\cdots\boldsymbol{L}_{k}^{-\left(1+\varepsilon\right)}}} & =0,
	\end{align}
	uniformly on $D$. 
	
	Let $\rho_{\beta,\varepsilon,k}^\pm$ and $M_{\varepsilon,k}$ be as defined in \eqref{eq:rho}. By \eqref{eq:asylog}, there exists $R>\exp^{\circ k}(0)$
	such that $\rho _{\beta , \varepsilon ,k}^{-}(R)>0$, $\rho _{\beta , \varepsilon ,k}^{-}$ is increasing on $[R,+\infty)$ and, for all $\zeta\in D_R$,
	\begin{equation}
	\left|f(\zeta)-\left(\zeta+\beta \right)\right|\leq\frac{1}{\left|\zeta\boldsymbol{L}_{1}\cdots\boldsymbol{L}_{k}^{1+\varepsilon}\right|}.\label{eq:new}
	\end{equation}
	Since $R>\exp^{\circ k}(0)$
	and $\left|\log\zeta\right|\geq\log\left|\zeta\right|\geq\log\left(\left|\Re\left(\zeta\right)\right|\right)=\log\left(\Re\left(\zeta\right)\right)$,
	we inductively get:
	\begin{align}\label{eq:new1}
		|\boldsymbol{L}_{m}| & \ge\log^{\circ m}(\Re\left(\zeta\right)),\text{ for }1\leq m\leq k,\ \zeta\in D_R.
	\end{align}
	Now, by \eqref{eq:new} and \eqref{eq:new1}, we get, for $\zeta\in D_R$:
	\begin{equation}
	|f(\zeta)-\left(\zeta+\beta \right)|\leq\frac{1}{\Re\left(\zeta\right)\cdot\log(\Re\left(\zeta\right))\cdots(\log^{\circ k}(\Re(\zeta)))^{1+\varepsilon}}.\label{PropEq1}
	\end{equation}
	Therefore, for $\zeta\in D_R$:
	\begin{align}
		\Re(f(\zeta))-\Re(\zeta) & \ge \Re(\beta)  -\frac{1}{\Re\left(\zeta\right)\cdot\log\left(\Re\left(\zeta\right)\right)\cdots(\log^{\circ k}(\Re(\zeta))){}^{1+\varepsilon}} \nonumber \\
		& =\rho_{\beta ,\varepsilon,k}^{-}\left(\Re\left(\zeta\right)\right) , \label{EqRealPart}
	\end{align}
	
	\begin{align}
		\Re(f(\zeta))-\Re(\zeta) & \leq \Re(\beta)  +\frac{1}{\Re\left(\zeta\right)\cdot\log\left(\Re\left(\zeta\right)\right)\cdots(\log^{\circ k}(\Re(\zeta))){}^{1+\varepsilon}}\nonumber\\
		& =\rho_{\beta ,\varepsilon,k}^{+}\left(\Re\left(\zeta\right)\right) ,\label{EqRealPart1}
	\end{align}
	and

	\begin{align}
		\Im(f(\zeta))-\Im(\zeta) & \geq \Im(\beta) -\frac{1}{\Re\left(\zeta\right)\cdot\log\left(\Re\left(\zeta\right)\right)\cdots(\log^{\circ k}(\Re(\zeta)))^{1+\varepsilon}}\label{eq:new3}\\
		& =\Im(\beta) - M_{\varepsilon,k}(\Re\left(\zeta\right)), \nonumber 
	\end{align}
	\begin{align}
		\Im(f(\zeta))-\Im(\zeta) & \leq \Im(\beta) +\frac{1}{\Re\left(\zeta\right)\cdot\log\left(\Re\left(\zeta\right)\right)\cdots(\log^{\circ k}(\Re(\zeta)))^{1+\varepsilon}}\label{eq:new2}\\
		& =\Im(\beta) + M_{\varepsilon,k}(\Re\left(\zeta\right)). \nonumber 
	\end{align}
	Since $\rho_{\beta,\varepsilon,k}^-$ is an increasing function, it follows that  $\rho_{\beta ,\varepsilon,k}^{-}(\Re\left(\zeta\right))\geq\rho_{\beta ,\varepsilon,k}^{-}(R)>0$, for every $\zeta\in D_R$. 
	Let, for $\zeta\in D_R$,
	\begin{align*}
		\mathcal{S}_{\beta ,\varepsilon,k}(\zeta) & :=\left[ \Re(\zeta) + \rho _{\beta , \varepsilon ,k}^{-}(\Re(\zeta)),\, \Re(\zeta)+\rho_{\beta ,\varepsilon,k}^{+}(\Re(\zeta))\right]\\
		& \times\left[\Im (\zeta )+\Im(\beta) -M_{\varepsilon,k}(\Re(\zeta)),\Im (\zeta )+\Im(\beta) +M_{\varepsilon,k}(\Re(\zeta))\right].
	\end{align*}
	By \eqref{EqRealPart}-\eqref{eq:new2}, we get that, for $\zeta\in D_R$,
	\[
	f(\zeta)\in\mathcal{S}_{\beta ,\varepsilon,k}(\zeta).
	\]
	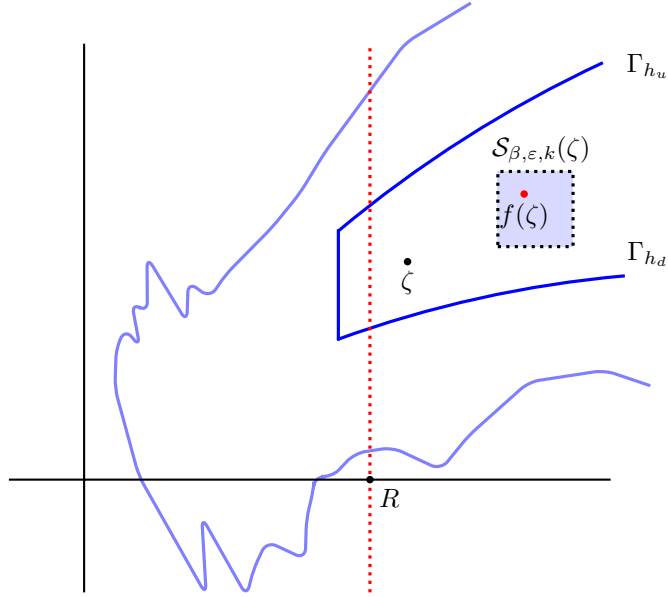
\begin{figure}[h!]
		\centering
		\begin{tikzpicture}
		\draw[blue!50!white,very thick,rounded corners,rotate=20] (7.5,-1.4) -- (7,-1) -- (6,-0.8) -- (5,-1.2) -- (4.5,-1.5) -- (4,-1) -- (3.5,-0.9) -- (3.1,-1.1) -- (2.9,-1) -- (2.8,-1.2) -- (2.6,-1.5) -- (2.4,-2) -- (2.2,-1.8) -- (2,-1.6) -- (1.7,-1.9) -- (1.5,-2.2) -- (1.3,-1) -- (1,-2) -- (0.7,-0.2) -- (0.8,1) -- (1,1.5) -- (1.2,1.7) -- (1.3,1.4) -- (1.35,1.7) -- (1.4,2) -- (1.55,1.8) -- (1.8,2.5) -- (2,1.5) -- (2.2,2) -- (2.4,1.7) -- (2.7,1.92) -- (3.5,2.3) -- (4,2.6) -- (4.5,3) -- (6,4) -- (7,4.2);
		\filldraw[draw=blue,very thick,fill=white,rotate=20] (4.3,1.95) arc (110:95:16cm);
		\filldraw[draw=blue,very thick,fill=white,rotate=20] (3.8,0.6) arc (90:75:15cm);
		\draw[blue,very thick] (3.38,1.85) -- (3.38,3.33);
		\draw[red,very thick,dotted] (3.8,-1.5) -- (3.8,5.8);
		\filldraw[draw=black,very thick,fill=blue!15!white,dotted] (5.5,3.1) rectangle (6.5,4.1);
		\draw[black,thick] (-1,0) -- (7,0) (0,-1.5) -- (0,5.8);
		\draw (4.3,2.9) node[draw=none,fill=black,circle,scale=0.3]{} (4.3,2.9) node[below]{$\zeta $} (5.85,3.8) node[below]{$f(\zeta )$} (5.85,3.8) node[draw=none,fill=red,circle,scale=0.3]{} (3.8,0) node[below right]{$R$} (3.8,0) node[draw=none,fill=black,circle,scale=0.3]{} (5.3,4.7) node[below right]{$\mathcal{S}_{\beta ,\varepsilon ,k}(\zeta )$} (7.5,5.5) node{$\Gamma _{h_{u}}$} (7.5,3) node{$\Gamma _{h_{d}}$};
		\end{tikzpicture}
		\caption{For $R$ sufficiently large, $\mathcal{S}_{\beta ,\varepsilon,k}(\zeta)\subseteq\overline{D}_{R}$,
			$\zeta\in\overline{D}_{R}$.}
		\label{fig:pict1}
	\end{figure}

	Now take $\zeta \in \overline{D}_{R}$. It is left to prove that then $f(\zeta)\in \overline{D}_R$, that is, that $\overline D_R$ is $f$-invariant. By definition of $\overline D_R$, there exists a $(\beta,\varepsilon,k)$-domain $(D_{h_{l},h_{u}})_R\subseteq\overline{D}_R$,
	such that $\zeta\in (D_{h_{l},h_{u}})_R$. Now, by properties $(2)$ and $(3)$ in the definition of lower-upper pair of type $(\beta, \varepsilon ,k)$, it
	follows that $\mathcal{S}_{\beta ,\varepsilon,k}(\zeta)\subseteq D_{h_{l},h_{u}}$, see Figure~\ref{fig:pict1}. 
	Therefore, $f(\zeta)\in D_{h_{l},h_{u}}\subseteq\overline{D}$. Since $\zeta\in D_R$ and $\rho_{\beta ,\varepsilon,k}^{-}(\Re(\zeta))>0$ for $\zeta\in D_R$, by \eqref{EqRealPart} it follows that $\Re(f(\zeta ))>\Re(\zeta)\geq R$. Therefore, $f\in \overline D_R$. 
\end{proof}
\medskip

We now state the main result of this section, which establishes
the convergence of Koenigs' sequence for a holomorphic map with a logarithmic-type bound on asymptotic behavior on an admissible domain: 
\begin{namedthm}
	{Theorem A} Let $\beta\in\mathbb C^+$, $\varepsilon>0$, $k\in\mathbb{N}$. Let $D\subseteq\mathbb C^+$
	be an admissible domain of type $(\beta ,\varepsilon,k)$. For $C>\exp^{\circ k}(0)$, let $f:D_C\to\mathbb C$ 
	be an analytic map such that
	\begin{equation}\label{eq:condit}
	f(\zeta)=\zeta+\beta +o(\zeta^{-1}\boldsymbol{L}_{1}^{-1}\cdots\boldsymbol{L}_{k}^{-\left(1+\varepsilon\right)}),\ 
	\text { as }\mathrm{\Re}(\zeta)\to +\infty \text{ uniformly on $D_C$}.
	\end{equation}
	Here, the iterated logarithms $
	\boldsymbol{L}_{1},\ldots \boldsymbol{L}_k$ are defined as in Proposition~\ref{prop:invariance of D}. Then:
	\begin{enumerate}[(i), font=\textup, nolistsep,leftmargin=0.6cm]
		\item \emph{(Existence)} For a sufficiently large $R>\exp^{\circ k}\left(0\right)$ there exists
		an analytic linearizing map $\varphi$ on the $f$-invariant
		subdomain $D_{R}^f\subseteq D$. That is, $\varphi$ 
		satisfies 
		\begin{align}
			(\varphi\circ f)(\zeta) & =\varphi(\zeta)+\beta,\text{ for all }\zeta\in{D}_{R}^f.\label{eq:linera-1}
		\end{align}
		Moreover, $\varphi$ is the uniform limit on $D_R^f$ of the Koenigs sequence
		$$(f^{\circ n}-n\beta)_{n}.$$
		\item If $D_R^f\cap \{\zeta\in\mathbb C^+:\Im(\zeta)=0\}$ is $f$-invariant, 
		so is $\varphi
		$-invariant.
		\item {\em(Asymptotics)} The linearization $\varphi$ is \emph{tangent to identity}, i.e. $\varphi (\zeta)=\zeta+o(1)$, uniformly on $D_{R}^f\subseteq \mathbb C^{+}$, as $\Re(\zeta)\to +\infty$. \\
		In particular, $\varphi(\zeta)=\zeta +o(\boldsymbol{L}_{k}^{-\nu})$, 
		for every $\nu\in\left(0,\varepsilon\right)$, uniformly as $\Re(\zeta)\to +\infty$,
		on every subdomain $D_{h_{l},h_{u}}\subseteq D_R^f$
		such that $h_{l}(x)=O(x)$ and $h_{u}(x)=O(x),\ x\to +\infty$. 
		\item \emph{(Uniqueness)} Let $\psi:D_{1}\to\mathbb C$, be
		a linearization of $f$ on an $f$-invariant subset $D_{1}\subseteq D$, such that $\psi(\zeta)=\zeta+o(1)$ uniformly on $D_{1}$, as $\Re(\zeta)\to +\infty$. Then $\psi\equiv \varphi$ on
		$(D_{1})_{R}$. 
	\end{enumerate}
\end{namedthm} 
\begin{rem}\
	
	(1)	Theorem A can be seen as a generalization to complex domains, expressed in the logarithmic chart $\zeta=-\log(z)$, of the results proved in Dewsnap-Fisher \cite{dewsnap-fischer:convergence-koenigs-sequences} for real maps.
	
	(2)	Note that the condition \eqref{eq:condit} for linearizability is natural. It was indeed proven in \cite{prrs:normal_forms_hyperbolic_logarithmic_transseries} that a hyperbolic logarithmic formal transseries $\widehat f\in\mathcal L_k$ (i.e. with monomials in variables $z,\boldsymbol\ell_1,\ldots,\boldsymbol\ell_k$, where $\boldsymbol{\ell }_{1}:=-\frac{1}{\log z}$, and inductively $\boldsymbol{\ell }_{i}:=\boldsymbol{\ell }_{1}\circ \boldsymbol{\ell }_{i-1}$, for $2\leq i\leq k$), $k\in\mathbb N_{\geq 1},$ is formally linearizable if and only if the leading monomial of $\widehat f-\lambda z$ is a logarithmic monomial of type $o(z\boldsymbol{\ell}_{1}\cdots\boldsymbol{\ell}_{m}^{1+\varepsilon})$ for some $m\in\mathbb N_{\geq 1}$, $1\leq m\leq k$, and $\varepsilon>0$. On the other hand, the monomials up to $z\boldsymbol{\ell}_{1}\cdots\boldsymbol{\ell}_{k}$ (included) \emph{cannot} be eliminated from the normal form of $\widehat f$. 
	
\end{rem}

\begin{example*}[Examples for Theorem A]
	Although Theorem A may be applied to \emph{any} function satisfying the condition \eqref{eq:condit}, we illustrate how this theorem works on some examples of functions that have an expansion in the logarithmic scale.
	
	(1) Let $$f(\zeta)=\zeta+2+3\pi\mathrm{i}+\zeta^{-1}\boldsymbol L_1^{-2}+\zeta^{-2}\boldsymbol L_2^{2}+\frac{\mathrm{e}^{-\zeta}}{1-\mathrm{e}^{-\zeta}\boldsymbol L_1}.$$ By Proposition~\ref{prop:invariance of D}, $f$ is well-defined on a domain $D_R\subseteq\mathbb C^+$ where $D$ is an $(\beta,\varepsilon,2)$-admissible domain, with $\beta=2+3\pi\mathrm{i}$ and $\varepsilon>0$, and $R>\mathrm{exp}(\exp(0))={\rm e}$. Consider for example any \emph{standard quadratic domain} $D$ as in Example $(3)$, intersected with the right half plane $\{\zeta\in\mathbb C^{+}: \ \Re(\zeta)>R\}$. By Theorem A, for sufficiently large $R>{\rm e}$, $D_R$ is $f$-invariant and $f$ is linearizable by a tangent to the identity change $\varphi(\zeta)=\zeta+o(1)$, analytic on $D_R$:
	$$
	\varphi\circ f=\varphi+2+3\pi\mathrm{i} \text { on } D_R. 
	$$ 
	
	(2) Let
	\begin{align}\label{eq:eff}
		f(\zeta)=\zeta+2+3\pi\mathrm{i}+\boldsymbol L_1^{-1}+\zeta^{-1} \boldsymbol L_2^{-1} \boldsymbol  L_3+&4\zeta^{-1} \boldsymbol L_1^{-1}\boldsymbol L_2^{-1} \boldsymbol L_3+\nonumber\\
		&+\zeta^{-2}+\mathrm{e}^{-\zeta}\boldsymbol L_1^2+\mathrm{e}^{-2\zeta},
	\end{align}
	on some $D_R$, where $D$ is a $(\beta,\varepsilon,3)$-admissible domain with $\beta=2+3\pi\mathrm{i}$, $\varepsilon>0$, and $R>\mathrm{exp}^{\circ 3}(0)$ sufficiently big so that $\overline{D}_R$ is invariant for $f$.
	
	By \cite{prrs:normal_forms_hyperbolic_logarithmic_transseries}\footnote{Written in the $z$-chart, $z=\mathrm{e}^{-\zeta}$, $f$ given in \eqref{eq:eff} is a series of monomials in $z,\boldsymbol\ell_1,\boldsymbol\ell_2$, $\boldsymbol\ell_3$ of the form $\tilde f(z):=\mathrm{e}^{-f(-\log z)}=\mathrm{e}^{-2-3\pi \mathrm{i}} z(1+(\mathrm{e}^{-\boldsymbol\ell_2}-1)-\boldsymbol\ell_1\boldsymbol\ell_3\boldsymbol\ell_4^{-1}-3\boldsymbol\ell_1\boldsymbol\ell_2\boldsymbol\ell_3\boldsymbol\ell_4^{-1})+o(z\boldsymbol\ell_1\boldsymbol\ell_2\boldsymbol\ell_3\boldsymbol\ell_4^{1+\varepsilon})$, $\varepsilon>0$. Also, in the $z$-chart, $f_0$ from \eqref{eq:nf} is of the form $\tilde f_0(z):=\mathrm{e}^{-f_{0}(-\log z)}=\mathrm{e}^{-2-3\pi\mathrm{i}} z(1+(\mathrm{e}^{-\boldsymbol\ell_2}-1)-\boldsymbol\ell_1\boldsymbol\ell_3\boldsymbol\ell_4^{-1}-3\boldsymbol\ell_1\boldsymbol\ell_2\boldsymbol\ell_3\boldsymbol\ell_4^{-1})+o(z\boldsymbol\ell_1\boldsymbol\ell_2\boldsymbol\ell_3\boldsymbol\ell_4^{1+\varepsilon})$. By \cite{prrs:normal_forms_hyperbolic_logarithmic_transseries}, $\tilde f$ can be formally reduced to $\tilde f_0$; therefore, $f$ can be formally reduced to $f_0$ (by the same conjugacy written in the logarithmic chart). On the other hand, suppose that $f_0$ from \eqref{eq:nf} can be further reduced, or any of its coefficients changed, to, say, $ f_1$. This would imply that, also in $z$-chart, $\tilde f$ can be normalized to $\tilde f_1(z):=\mathrm{e}^{- f_1(-\log z)}$. However, it is easy to check that $\tilde f_1$ and $\tilde f_0$ here defined differ in terms before the monomial $z\boldsymbol\ell_1\boldsymbol\ell_2\boldsymbol\ell_3\boldsymbol\ell_4^{1+\varepsilon}$, which is a contradiction with normalization result in \cite{prrs:normal_forms_hyperbolic_logarithmic_transseries}, which states that terms before $z\boldsymbol\ell_1\boldsymbol\ell_2\boldsymbol\ell_3\boldsymbol\ell_4^{1+\varepsilon}$ in $\tilde f_0$ are not affected by formal changes of variables.}, if we consider $f$ as a \emph{formal} series in variables $\zeta^{-1},\boldsymbol L_1,\boldsymbol L_2,\boldsymbol L_3$ and $\mathrm{e}^{-\zeta}$, $f$ is not \emph{formally} linearizable by any logarithmic transseries. Nevertheless, it can be formally reduced to a  normal form
	\begin{equation}\label{eq:nf}
		f_0(\zeta)=\zeta+2+3\pi\mathrm{i}+\boldsymbol L_1^{-1}+\zeta^{-1} \boldsymbol L_2^{-1} \boldsymbol L_3+4\zeta^{-1} \boldsymbol L_1^{-1}\boldsymbol L_2^{-1} \boldsymbol L_3,
	\end{equation}
	where the terms strictly before the term $\zeta^{-1}\boldsymbol L_1^{-1}\boldsymbol L_2^{-1}\boldsymbol L_3^{-1-\varepsilon}$, for any $\varepsilon>0$, cannot be eliminated from the normal form by formal changes of variables.
	
	Note that Theorem A does not claim the analytic linearization of $f$ on $\overline{D}_R$, since $f$ does not satisfy \eqref{eq:condit}. Although Theorem A is not an ``if and only if'' statement, formal non-linearizability is a good indication that analytic linearization of $f$ on $\overline{D}_R$ may not be possible. We conjecture instead that the proof of Theorem~A can be adapted to the analytic normalization of $f$ to $f_0$ on $\overline{D}_R$ in this case. We omit it here, since the use of techniques similar to the proof of Theorem A is to be expected, but with more technical complications.
\end{example*}

\begin{proof}[Proof of Theorem A]\
	
	(i) By Proposition \ref{prop:invariance of D},
	there exists $R>\exp^{\circ k}\left(0\right)$ sufficiently large, such
	that $\overline{D}_{R}\subseteq{D}_{R}^f$, where $D^f$
	is the maximal $f$-invariant subdomain of $D\subseteq\mathbb C^+$.
	Therefore, ${D}_{R'}^f\neq\emptyset$, for all $R'\ge R$.
	Let $\zeta\in{D}_{R}^f$ and let $\rho_{\beta,\varepsilon,k}^\pm$ be the increasing functions defined in \eqref{eq:rho}. From \eqref{EqRealPart}, for $\zeta\in D_R^f$ with $R$ sufficiently large, since $\rho_{\beta,\varepsilon,k}^-$ is increasing on $[R,+\infty)$, it follows that:
	\begin{equation}\label{Equat1}
	\Re(f^{\circ n}(\zeta))\geq \Re\left(\zeta\right)+n\rho_{\beta ,\varepsilon,k}^{-}(\Re\left(\zeta\right))\geq R+n\rho_{\beta,\varepsilon,k}^{-}(R),\ n\in\mathbb N_{\geq 1}.
	\end{equation}
	By \eqref{PropEq1}, for $\zeta\in D_R^f$ it holds that:
	\begin{equation}\label{Equati2}\left|f(\zeta)-\left(\zeta+\beta \right)\right|\leq\frac{1}{\Re\left(\zeta\right)\cdot\log(\Re\left(\zeta\right))\cdots(\log^{\circ k}(\Re\left(\zeta\right)))^{1+\varepsilon}}=M_{\varepsilon,k}(\Re\left(\zeta\right)).
	\end{equation}
	From \eqref{Equat1} and \eqref{Equati2} and since $M_{\varepsilon,k}$ is decreasing on $[R,+\infty)$, we inductively obtain, for $n\in\mathbb N_{\geq 1}$ and $\zeta\in D_R^f$:
	\begin{align}\label{eq:joj}
		\left|f^{\circ(n+1)}(\zeta)-\left(n+1\right) \beta -\left(f^{\circ n}(\zeta)-n\beta \right)\right| & =\left|f\left(f^{\circ n}(\zeta)\right)-\left(f^{\circ n}(\zeta)+\beta \right)\right| \nonumber \\
		& \leq M_{\varepsilon,k}\bigl(\Re(f^{\circ n}(\zeta))\bigr)\nonumber \\
		& \leq M_{\varepsilon,k}\left(\Re\left(\zeta\right)+n\rho_{\beta,\varepsilon,k}^{-}(\Re\left(\zeta\right))\right)\nonumber \\
		& \leq M_{\varepsilon,k}\left(R+n\rho_{\beta ,\varepsilon,k}^{-}(R)\right).
	\end{align}
	As stated in Subsection~\ref{subsec:admi}, the series $\sum_{n\geq0}M_{\varepsilon,k}\big(R+n\rho_{\beta,\varepsilon,k}^{-}(R)\big)$
	converges. Therefore, the Koenigs sequence $(f^{\circ n}-n\beta )_{n}$ is uniformly Cauchy, hence converges uniformly on $D_R^f$. Denote by $\varphi$ its  uniform limit on the domain ${D}_{R}^f$. By Weierstrass' theorem, it follows that
	$\varphi$ is analytic on  ${D}_{R}^f$. 
	
	Finally, we compute:
	\begin{align*}
		(\varphi\circ f)(\zeta) & =\lim_{n}\left(f^{\circ n}(f(\zeta))-n \beta\right)\\
		& =\lim_{n}\left(f^{\circ(n+1)}(\zeta)-(n+1)\beta \right)+\beta \\
		& =\varphi(\zeta)+\beta .
	\end{align*}
	Therefore, $\varphi$ is an analytic linearization of $f$ on $D_R^f$, obtained as a uniform limit of the Koenigs sequence. \\
	
	(ii) Suppose that $D_R^f\cap \{\zeta\in\mathbb C: \Im(\zeta)=0\}$ is invariant under $f(\zeta)=\zeta+\beta+o(1)$. Then obviously $\Im(\beta)=0$. Now consider the pointwise limit 
	\begin{equation}\label{eq:fi}
	\varphi(\zeta):=\lim_{n\to\infty}(f^{\circ n}(\zeta)-n\beta),\ \zeta \in D_R^f\cap \{\zeta\in\mathbb C: \Im(\zeta)=0\}.
	\end{equation}
	Since $D_R^f\cap \{\zeta\in{\mathbb C}^+: \Im(\zeta)=0\}$ is invariant for (all iterates of) $f$, 
	$\{\zeta \in \mathbb{C} : \Im(\zeta) =0\}$
	closed in $\mathbb{C}^+$, and since $\Im(\beta)=0$, \eqref{eq:fi} implies that $\{D_R^f\cap \{\zeta\in\mathbb C: \Im(\zeta)=0\}$ is invariant for $\varphi$. \\
	
	(iii) For $0<\nu<\varepsilon$, $\zeta\in{D}_{R}^f$ and $m\in\mathbb{N}_{\geq1}$, taking the sum of the terms
	\[f^{\circ(n+1)}(\zeta)-\left(n+1\right) \beta -\left(f^{\circ n}(\zeta)-n\beta \right)
	\]
	in \eqref{eq:joj} for $n$ ranging from $0$ to $m-1$ 
	it follows that
	\begin{align}
		\left|f^{\circ(m)}(\zeta)-m\beta -\zeta\right|&\leq\sum_{n=0}^{m-1}M_{\varepsilon,k}\left(\Re\left(\zeta\right)+n\rho_{\beta ,\varepsilon,k}^{-}(\Re\left(\zeta\right))\right) \nonumber \\
		& =\sum_{n=0}^{m-1}\frac{M_{\nu,k}\left(\Re\left(\zeta\right)+n\rho_{\beta ,\varepsilon,k}^{-}(\Re\left(\zeta\right))\right)}{\left(\log^{\circ k}\left(\Re\left(\zeta\right)+n\rho_{\beta ,\varepsilon,k}^{-}\left(\Re\left(\zeta\right)\right)\right)\right)^{\varepsilon-\nu}}  \nonumber \\
		& \leq\frac{1}{\left(\log^{\circ k}\left(\Re\left(\zeta\right)\right)\right)^{\varepsilon-\nu}}\cdot\sum_{n=0}^{+\infty}M_{\nu,k}(R+n\rho_{\beta , \varepsilon,k}^{-}(R))\nonumber\\
		&\leq \frac{C}{\left(\log^{\circ k}\left(\Re\left(\zeta\right)\right)\right)^{\varepsilon-\nu}},\ C>0,\label{eqConvergence}
	\end{align}
	where the last sum converges to $C>0$. Taking the pointwise limit for ${m\to+\infty}$ in \eqref{eqConvergence}, it follows that 
	\begin{equation}\label{eq:hahh}
	\left|\varphi(\zeta) -\zeta\right|\leq \frac{C}{\left(\log^{\circ k}\left(\Re\left(\zeta\right)\right)\right)^{\varepsilon-\nu}},\ C>0,\ \zeta\in D_R^f.
	\end{equation}
	Therefore, 
	$\varphi(\zeta )=\zeta + o(1)$, uniformly on $D_{R}^f$ as $\Re (\zeta )\to + \infty $. 
	\medskip
	
	To get a more rigorous estimate, using the elementary properties of the logarithm, we easily see that,
	for every domain $D_{h_{l},h_{u}}\subseteq\mathbb C_+$
	such that $h_{l}(x)=O(x)$ and $h_{u}(x)=O(x)$, as $x\to+\infty$, there exists $N>0$, such that
	\begin{align}
		|\boldsymbol{L}_{k}| & \leq N\log^{\circ k}\left(\Re\left(\zeta\right)\right),\text{ for }\zeta\in (D_{h_{l},h_{u}})_R, \label{eqLogarit}
	\end{align}
	for sufficiently large $R>\exp^{\circ k}\left(0\right)$, $k\in\mathbb N_{\geq 1}$. Indeed, the conditions $h_{l}(x)=O(x)$ and $h_{u}(x)=O(x)$ imply that there exists $d>0$ such that $|\Im( \zeta)|\leq d\cdot  \Re(\zeta),\ \zeta\in (D_{h_l,h_u})_R$, for a sufficiently large $R>0$.
	
	Using \eqref{eqLogarit} in \eqref{eq:hahh}, it now follows that, for any $0<\nu<\nu'<\varepsilon$, there exists $E>0$ such that:
	\begin{equation}
	\left|\varphi(\zeta)-\zeta\right|\leq \frac{E}{\left(\log^{\circ k}\left(\Re\left(\zeta\right)\right)\right)^{\nu'-\nu}}|\boldsymbol{L}_k|^{\nu'-\varepsilon},\ \zeta\in (D_{h_{l},h_{u}})_R.
	\end{equation}
	This implies that, for every $0<\nu'<\varepsilon,$ on $(D_{h_{l},h_{u}})_R$ it holds that:
	\begin{align*}
		\lim_{\Re (\zeta) \to +\infty}\frac{\varphi(\zeta)-\zeta}{\frac{1}{\left(\boldsymbol{L}_{k}\right)^{\varepsilon-\nu'}}} & =0.
	\end{align*}
	Therefore, for any $0<\nu<\varepsilon$, $\varphi(\zeta)=\zeta+o\left(\boldsymbol{L}_{k}^{-\nu}\right)$, as $\Re(\zeta)\to +\infty$ in $(D_{h_{l},h_{u}})_R$. 
	\\
	
	(iv) Suppose that $\psi$ is an analytic linearizing germ,
	i.e. $\psi\circ f=\psi+\beta$, on $f$-invariant subset $D_{1}\subseteq D$, such that $\psi (\zeta)=\zeta+o(1)$ uniformly on $D_{1}$ as $\Re(\zeta)\to+\infty$. Since $D_1$ is $f$-invariant and $D^f$ is a maximal $f$-invariant subdomain of $D$, obviously $(D_{1})_{R}\subseteq {D}_{R}^f$. Clearly, $(D_1)_R$ is also $f$-invariant, and by \eqref{Equat1}, non-empty. Recall from (i) that $\varphi$ is the analytic linearization constructed on whole $D_{R}^f$ as the limit of the Koenigs sequence, for sufficiently large $R>\exp ^{\circ k}(0)$. It satisfies $\varphi(\zeta)=\zeta+o(1)$, uniformly as $\Re(\zeta)\to+\infty$ on $D_R^f$. We now show that $\psi\equiv\varphi$ on $(D_1)_R$.
	
	Put $$E(\zeta):=\varphi(\zeta)-\psi(\zeta),\ \zeta\in (D_{1})_{R}.$$ Then $E$ is analytic on $(D_{1})_{R}$ and $E(\zeta)=o(1)$, as $\Re(\zeta)\to +\infty$ uniformly on $(D_{1})_{R}$.
	Moreover, $(E\circ f)(\zeta)=E(\zeta)$, $\zeta\in (D_{1})_{R}$.
	Inductively, since $(D_1)_R$ is $f$-invariant, we obtain \begin{equation}\label{eq:contr}E(f^{\circ n}(\zeta))=E(\zeta),\ \zeta\in (D_{1})_{R},\ n\in\mathbb{N}.\end{equation} By \eqref{Equat1}, $\Re\left(f^{\circ n}(\zeta)\right)\geq R+n\rho_{\beta ,\varepsilon,k}^{-}(R)$, for $n\in \mathbb{N}$ and $\zeta\in (D_1)_R\subseteq D_R^f$. It follows that\begin{equation}\label{eq:dva} \lim_{n}\Re\left(f^{\circ n}(\zeta)\right)=+\infty,\ \zeta\in (D_1)_R.\end{equation}
	Passing to limit, as $n\to\infty$, in \eqref{eq:contr}, and using \eqref{eq:dva} and the fact that $E(\zeta)=o(1)$, as $\Re(\zeta)\to +\infty$,
	we get that $E(\zeta)=0,$ for every $\zeta\in (D_1)_R$. That is, $\varphi\equiv\psi$ on $(D_1)_R$.
\end{proof}

\section{\label{sec:linearization Dulac maps}Linearization of Dulac maps}

This final section is dedicated to the proof of our main result (Theorem
B, Section \ref{sec:notation main result}).

\subsection{Linearization of a hyperbolic Dulac series}

We prove in this subsection that a hyperbolic Dulac series \eqref{eq:Dulac}
admits a unique (formal) parabolic Dulac linearization. A key argument in the proof is the linearization of \emph{hyperbolic logarithmic transseries} established in \cite{prrs:normal_forms_hyperbolic_logarithmic_transseries}. Hence we first recall the notation used in \cite{prrs:normal_forms_hyperbolic_logarithmic_transseries}. Notice that, while the results there where established for logarithmic transseries with \emph{real coefficients}, we need here to consider certain logarithmic transseries with \emph{complex coefficients}. For the series considered here, as the proof consists of algebraic computations on the coefficients, switching from real to complex coefficients does not affect the results of \cite{prrs:normal_forms_hyperbolic_logarithmic_transseries}.   

\begin{notation}
	In this section, $z$ denotes an infinitesimal variable and the symbol $\boldsymbol{\ell}_{1}$ denotes the transmonomial $-\dfrac{1}{\log z}$.
\end{notation}

Following \cite{prrs:normal_forms_hyperbolic_logarithmic_transseries}, we describe a class of \emph{logarithmic transseries} in the variable $z$.
We denote by $\mathcal{L}_{1}(\mathbb C)$ the collection of logarithmic transseries
of the form
\begin{equation}
	\widehat{f}_1=\sum_{\left(\alpha,n\right)\in\mathbb{R}_{\geq0}\times\mathbb{Z}}
	a_{\alpha,n}\cdot z^{\alpha}\boldsymbol{\ell}_{1}^{n},
	\label{eq:series L1}
\end{equation}
where $\mathrm{Supp}(\widehat{f}_1):=\left\{ \left(\alpha,n\right)\in\mathbb{R}_{\geq0}\times\mathbb{Z}:a_{\alpha,n}\ne0\right\} $,
called the \emph{support} of $\widehat{f}_1$, is a \emph{well-ordered} subset
of $\mathbb{R}_{\geq0}\times\mathbb{Z}$ (for the lexicographic order),
and $a_{\alpha,n}\in\mathbb{C}$.   

For $\widehat{f}_1\in\mathcal{L}_{1}(\mathbb C)$, let $\text{ord}\left(\widehat{f}_{1}\right):=\min \, \mathrm{Supp}\left(\widehat{f}_{1}\right) \in\mathbb{R}_{\ge0}\times\mathbb{Z}$ if $\widehat{f}_1\ne0$ and $\mathrm{ord}(0)=+\infty$.
Let $\mathrm{ord}_{z}(\widehat{f}_1)$ be the smallest $\alpha\in\mathbb{R}_{\ge0}$
such that there exists $n\in\mathbb{Z}$ with $\left(\alpha,n\right)\in\mathrm{Supp}(\widehat{f}_1)$ if $\widehat{f}_1\neq0$ and $\mathrm{ord}_z(0)=+\infty$.
We also use the acronym $\mathrm{h.o.t.}$ for ``higher order terms''
when we describe a transseries.

Consider a transseries $\widehat{f}_1\in\mathcal{L}_{1}(\mathbb C)$ such that $\mathrm{ord}\left(\widehat{f}_{1}\right)=\left(1,0\right)$. If $a_{1,0}=1$, the transseries $\widehat{f}_1$ is called {\em parabolic}. If $|a_{1,0}|\ne0,1$, the series $\widehat{f}_1$ is called {\em hyperbolic}.

An element of $\mathcal{L}_{1}(\mathbb C)$ can also be written \emph{blockwise}:
\begin{equation}
\widehat{f}_1=\sum_{\alpha\in\mathbb{R}_{\geq0}}z^{\alpha}R_{\alpha}(\boldsymbol{\ell}_{1}),
\label{eq:write block}
\end{equation}
where each $R_{\alpha}$, called the \emph{block of
	index $\alpha$}, is an element of the field of Laurent series $\mathbb{C}\left((\boldsymbol{\ell}_{1})\right)$. 

To a parabolic transseries $\widehat{f}_{1}\left(z\right)$ as in \eqref{eq:write block} corresponds
the \emph{parabolic complex exponential transseries} in the infinite variable $\zeta=-\log z$
\begin{equation}
\widehat{f}\left(\zeta\right)=-\log\left(\widehat{f}_{1}\left(\exp\left(-\zeta\right)\right)\right)=\zeta+\sum_{\nu\ge0}\exp\left(-\nu\zeta\right)P_{\nu}\left(\zeta^{-1}\right),
\label{eq:parabolic-exponential}
\end{equation}
where, for each $\nu$, $P_{\nu}\in\mathbb{C}((X))$ is a Laurent series, and $\varepsilon\left(\zeta\right):=\widehat{f}\left(\zeta\right)-\zeta$
is an infinitesimal transseries. \\

A transseries in $\mathcal L_{1}(\mathbb{C})$ might fail to be a logarithmic Dulac series \eqref{eq:log-dulac-series} for two reasons:

1. The exponents $\alpha$ form a well-ordered subset of $\mathbb{R}_{>0}$,
but in general they do not belong to a finitely generated sub-semigroup
of $\mathbb{R}_{>0}$. Moreover, they might not form a strictly increasing
sequence tending to $+\infty$.

2. The blocks $R_{\alpha}$ are Laurent series in $\boldsymbol{\ell}_{1}$,
but they are not necessarily complex \emph{polynomials} in $\boldsymbol{\ell}_{1}^{-1}$ (additionally, in a Dulac series, the polynomial in the leading block must necessarily be a constant).\\

These are precisely the two properties that have to be checked to
guarantee that an element of $\mathcal{L}_{1}(\mathbb C)$ is a complex Dulac series,
as in the proof of the following lemma.
\begin{lem}[Formal linearization]
	\label{lem:formal dulac linearization} Let $\widehat{f}= \zeta+\beta+\mathrm{h.o.t.}$,
	$\beta\in\mathbb{C}^+$, be a hyperbolic complex Dulac series \eqref{eq:log-dulac-series}. Then there exists a unique parabolic complex exponential transseries $\widehat{\varphi}$ such that $\widehat{\varphi}\circ\widehat{f}=\widehat{\varphi}+\beta$. Moreover, $\widehat\varphi$ is a complex Dulac transseries as in  \eqref{eq:Dulac}. Finally, if the coefficients of $\widehat{f}$ are real, then so are those of $\widehat{\varphi}$.
\end{lem}
\begin{proof}
	\emph{Existence.} Let $\lambda:=\exp(-\beta)\in\mathbb{C}$, where $\exp$ is the complex exponential function and not the compositional inverse of the logarithmic chart. Let
\begin{align*}
	\widehat{f}_1\left(z\right) & :=\exp\left(-\widehat{f}\left(-\log z\right)\right)\\
	& =\exp\left(-\left(-\log z+\beta+\sum_{i=1}^{\infty}\exp\left(\nu_{i}\log z\right)R_{i}\left(-\log z\right)\right)\right),\thinspace R_{i}\in\mathbb{C}\left[X\right], \, \nu _{i} >0 , \\
	& =\lambda z\exp\left(\sum_{i=1}^{\infty}z^{\nu_{i}}R_{i}\left(-\log z\right)\right)\\
	& =\lambda z+\sum_{i=1}^{\infty}z^{\alpha_{i}}P_{i}\left(\log z\right),\quad \alpha_i>1,\thinspace P_i\in\mathbb{C}[X],
\end{align*}
so that $\widehat{f}_1\in\mathcal{L}_{1}\left(\mathbb{C}\right)$ is
a complex logarithmic Dulac series in the variable $z$. Notice that here $\lambda$
is a complex number, and is not seen as the element of $\widetilde{\mathbb{C}}$ parameterized by $\beta$ in the logarithmic chart.
It is indeed important that all the coefficients of $\widehat{f}_1$
are complex numbers and not elements of $\widetilde{\mathbb{C}}$, in
order to apply to them all the algebraic computations involved in
the proof of the Main Theorem of \cite{prrs:normal_forms_hyperbolic_logarithmic_transseries}.
  	
	The latter implies that $\widehat{f}_1$ admits a unique parabolic linearization 
$\widehat{\varphi}_1\in\mathcal{L}_{1}(\mathbb C)$.
	Let $\widehat{g}_1:=\widehat{f}_1-\lambda\cdot\mathrm{id}$ and $\gamma:=\mathrm{ord}_{z}(\widehat{g}_1)$.
	As $\widehat{f}_1$ is a Dulac series, we have $\gamma>1$. Moreover,
	recall that the exponents of $z$ in $\widehat{f}_1$ form a finitely
	generated strictly positive sequence which tends to $+\infty$. Hence, we deduce from
	the description of the support of $\widehat{\varphi}_1$ given in 
	\cite[Section 5]{prrs:normal_forms_hyperbolic_logarithmic_transseries}
	that the exponents of $z$ in $\widehat{\varphi}_1$ also form a finitely
	generated strictly positive sequence which tends to $+\infty$. 
	
	We now prove the \emph{polynomial property} for blocks of the linearization $\widehat\varphi_1$:  that each monomial
	$z^{\alpha}$ in $\widehat{\varphi}_1$ is multiplied by a complex \emph{polynomial
		in $\boldsymbol{\ell}_{1}^{-1}=-\log z$}. By the proof of the Main Theorem
	in \cite{prrs:normal_forms_hyperbolic_logarithmic_transseries}, the
	linearization $\widehat{\varphi}_1$ is given by
	$$
	\widehat{\varphi}_1:=\mathrm{id}+\widehat{h}_1,
	$$
	where $\widehat{h}_1\in\mathcal{L}_{1}(\mathbb C)$, $\mathrm{ord}_{z}(\widehat{h}_1)>1$, 
	is the limit of the Picard sequence $\left(\widehat{\psi}_{n}\right)_{n\in\mathbb{N}}$
	defined by
	\begin{equation}
		\widehat{\psi}_{n}:=(\mathcal{T}_{\widehat{f}_1}^{-1}\circ\mathcal{S}_{\widehat{f}_1})^{\circ n}(0).\label{eq:pmt}
	\end{equation}
	Here, the limit is taken in the sense of the \emph{valuation topology} (see e.g. \cite{vdd_macintyre_marker:log-exp_series}): $\mathrm{ord}_{z}\left(\widehat{\psi}_{n}-\widehat{\varphi}_1\right)$
	tends to $+\infty$, as $n\to +\infty$. The operators $\mathcal{T}_{\widehat{f}_1}$
	and $\mathcal{S}_{\widehat{f}_1}$ (case $\beta>1$) are defined in \cite{prrs:normal_forms_hyperbolic_logarithmic_transseries}
	as
	\begin{equation}
		\mathcal{S}_{\widehat{f}_1}(\widehat{h})=\dfrac{1}{\lambda}\widehat{g}_1+\frac{1}{\lambda}\sum_{i\ge1}\frac{\widehat{h}^{\left(i\right)}\left(\lambda z\right)}{i!}\widehat{g}_1^{i}\text{,\ \ and\ \  }\mathcal{T}_{\widehat{f}_1}(\widehat{h})=\widehat{h}-\frac{1}{\lambda}\widehat{h}(\lambda z).\label{eq:definition Sf and Tf}
	\end{equation}
	In \eqref{eq:definition Sf and Tf}, transseries are applied to $\lambda z$. This means, using the following compositional rules, that
	$$
	\log(\lambda z)=\log(\lambda)+\log(z)=-\beta+\log(z)
	$$
	and, for $\alpha>0$,
	$$
	(\lambda z)^\alpha=\lambda^\alpha z^\alpha=\exp(\alpha\log(\lambda))z^\alpha=\exp(-\alpha\beta)z^\alpha.
	$$ 
In particular, we see that, in this proof, $\lambda$ is the only complex number for which we have to impose a determination of the logarithm. We chose $\log(\lambda)=-\beta$ in view of the final step of the proof, in which we deduce the linearization of $\widehat{f}$ from the linearization of $\widehat{f}_1$.

Now, due to the convergence of \eqref{eq:pmt} to the linearization $\widehat\varphi_1$ in the valuation topology, it suffices to prove the following: if $\widehat{h}\in\mathcal{L}_{1}(\mathbb C)$ with
	$\mathrm{ord}_{z}(\widehat{h})>1$ satisfies the polynomial property,
	the same holds for $\mathcal{T}_{\widehat{f}_1}^{-1}(\widehat{h})$
	and for $\mathcal{S}_{\widehat{f}_1}(\widehat{h})$. 
	
	Notice that the
	polynomial property is preserved under differentiation, under multiplication
	by a complex Dulac series, as well as under precomposition with $\lambda z$. Therefore, 
	if $\widehat h$ has the polynomial property, then so does $\mathcal{S}_{\widehat{f}_1}(\widehat{h})$, thanks to \eqref{eq:definition Sf and Tf}, the previous remark and  the
	fact that, as $\gamma>1$, $\mathcal{S}_{\widehat{f}_1}$ is an infinite
	sum of operators which strictly increase $\mathrm{ord}_{z}$. 
	
	Let us now check the polynomial property for $\mathcal{T}_{\widehat{f}_1}^{-1}(\widehat h)$. Suppose the contrary, that is, that there exists $\widehat{h}\in\mathcal{L}_{1}(\mathbb C)$,
	$\mathrm{ord}_{z}(\widehat{h})>1$, which satisfies the polynomial
	property, while $\mathcal{T}_{\widehat{f}_1}^{-1}(\widehat{h})$ does
	not. Then $\mathcal{T}_{\widehat{f}_1}^{-1}(\widehat{h})$ admits a block $\widehat{R}_{\nu}(\boldsymbol{\ell}_{1})\in\mathbb{C}\left(\left(\boldsymbol{\ell}_{1}\right)\right)$
	, for some $\nu>1$, which is not a polynomial in $\boldsymbol{\ell}_{1}^{-1}$.
	Hence we can write 
	\[
	\widehat{R}_{\nu}\left(\boldsymbol{\ell}_{1}\right)=Q\left(\boldsymbol{\ell}_{1}^{-1}\right)+a\left(\boldsymbol{\ell}_{1}\right),\text{ with }a\left(\boldsymbol{\ell}_{1}\right)=\sum_{n\ge n_{0}}a_{n}\boldsymbol{\ell}_{1}^{n},\ a_n\in\mathbb C,
	\]
	where $Q\in\mathbb{C}\left[\boldsymbol{\ell}_{1}^{-1}\right]$ is a polynomial and $a\in\mathbb C\left[\left[\boldsymbol{\ell}_{1}\right]\right]$
	is a nonzero power series such that  $a_{n_{0}}\neq 0$, $n_{0}\in \mathbb{N}_{\geq 1}$. Now apply
	$\mathcal{T}_{\widehat{f}_1}$ from \eqref{eq:definition Sf and Tf} to such $\mathcal{T}_{\widehat{f}_1}^{-1}(\widehat{h})$, to obtain $\widehat h$. However,
	using \eqref{eq:definition Sf and Tf} and the evident relations (see \cite[Section 3.6]{prrs:normal_forms_hyperbolic_logarithmic_transseries}) 
	\begin{align*}
	&\boldsymbol{\ell}_{1}^{-1}\left(\lambda z\right)=\boldsymbol{\ell}_{1}^{-1}-\log\lambda=\boldsymbol{\ell}_{1}^{-1}+\beta,\\
	&\boldsymbol{\ell}_{1}\left(\lambda z\right)=\boldsymbol{\ell}_{1}\cdot\left(1+\varepsilon\left(\boldsymbol{\ell}_{1}\right)\right),\ \text{for some }\varepsilon\in\mathbb{C}\left[\left[\boldsymbol\ell_1\right]\right],\ \varepsilon\left(0\right)=0,\\
	&\boldsymbol{\ell}_{1}^{n}\left(\lambda z\right)=\boldsymbol{\ell}_{1}^{n}\cdot\left(1+\varepsilon_{n}\left(\boldsymbol{\ell}_{1}\right)\right),\ \text{for some }\varepsilon_{n}\in\mathbb{C}\left[\left[\boldsymbol\ell_1\right]\right],\  \varepsilon_{n}\left(0\right)=0\ (n\in \mathbb{N}_{\geq 1}),
	\end{align*}
it is easy to see that the block of index $\nu$ in $\widehat{h}=\mathcal{T}_{\widehat{f}_1}\left(\mathcal{T}_{\widehat{f}_1}^{-1}(\widehat{h})\right)$ is
	\begin{equation}
		Q\left(\boldsymbol{\ell}_{1}^{-1}\right)+a\left(\boldsymbol{\ell}_{1}\right)-\lambda^{\nu-1}\left(Q\left(\boldsymbol{\ell}_{1}^{-1}+\beta\right)+\widetilde{a}\left(\boldsymbol{\ell}_{1}\right)\right),\label{eq:tje}
	\end{equation}
	where $\widetilde{a}\in\mathbb{C}\left[\left[\boldsymbol\ell_1\right]\right]$ is a power series such that $\widetilde{a}\left(0\right) =0$ and with $a_{n_0}\boldsymbol{\ell}_1^{n_0}$ as leading term.
	The power series $a-\lambda^{\nu-1}\widetilde{a}\in\mathbb{C}[[\boldsymbol{\ell}_1]]$ does not have a constant term, but is nonzero because
	its smallest coefficient is equal to $\left(1-\lambda^{\nu-1}\right)a_{n_{0}}\ne0$.
	This contradicts the fact that the block of index $\nu$ of
	$\widehat{h}$ is a polynomial in $\boldsymbol{\ell}_{1}^{-1}$. Therefore, $\widehat{\varphi }_{1}$ is a logarithmic Dulac series. \\

Finally, let $\widehat{\varphi}\left(\zeta\right):=-\log\left(\widehat{\varphi}_{1}\left(\exp\left(-\zeta\right)\right)\right)$.
Since we chose $\log(\lambda)=-\beta$, we have that  $\widehat{f}(\zeta)=-\log\left(\widehat{f}_{1}\left(\exp(-\zeta)\right)\right)$
and $-\log\left(\lambda\widehat{\varphi}_{1}\left(\exp\left(-\zeta\right)\right)\right)=\widehat{\varphi}\left(\zeta\right)+\beta$. Hence we deduce from $\widehat{\varphi}_{1}\circ\widehat{f}_{1}=\lambda\widehat{\varphi}_{1}$
that $\widehat{\varphi}\circ\widehat{f}=\widehat{\varphi}+\beta$.

Notice that in this proof, if the coefficients of $\widehat f$ are real, so are the coefficients of $\widehat{f}_1$, $\widehat{\varphi}_1$, and $\widehat{\varphi}$.\\

\emph{Uniqueness.} We have $\widehat{f}(\zeta)=\zeta+\beta+\widehat\varepsilon(\zeta)$, where $\widehat\varepsilon(\zeta)$ is an infinitesimal complex exponential transeries with the polynomial property. The difference $\widehat{\psi}$ between two parabolic complex linearizations of $\widehat{f}$ of type \eqref{eq:parabolic-exponential} satisfies
\begin{equation}
\widehat{\psi}\circ\widehat{f}=\widehat{\psi}.
\label{eq:uniqueness-equation}
\end{equation}
If $\widehat{\psi}\ne0$, then its leading term is $a\zeta^p\exp(-\nu\zeta)$ for some $a\in\mathbb{C}\setminus\{0\}$, $\nu\geq0$ and $p\in\mathbb{Z}$. There are two cases.

If $\nu>0$, then the leading term of $\widehat{\psi}\circ\widehat{f}$ is $a\zeta^p\exp(-\nu\zeta)\exp(-\nu\beta)$ (because $\exp(\varepsilon(\zeta))=1+\eta(\zeta)$, where $\eta$ is an infinitesimal transseries). It follows from \eqref{eq:uniqueness-equation} that $\exp(-\nu\beta)=1$, which is impossible, as $\Re(\beta)>0$ and $\nu>0$.

If $\nu=0$, then $p<0$. In that case, $\widehat{\psi}(\zeta)=a\zeta^p+b\zeta^{p-1}+\mathrm{h.o.t.}$, with $b\in\mathbb{C}$ possibly equal to $0$. Comparing the coefficients of $\zeta^{p-1}$ on both sides of \eqref{eq:uniqueness-equation}, we see that $pa\beta$ must be $0$, a contradiction.
\end{proof}



\subsection{Partial linearizations and homological equations}\label{SubSectionCohomol}
Consider a nontrivial hyperbolic complex Dulac map $f$ on a standard quadratic domain $\mathcal{R}_{C}\subset\mathbb{C}^+$
(\emph{nontrivial} meaning that $f\left(\zeta\right)\neq\zeta+\beta$, $\beta\in\mathbb{C}^+$). Then
$f$ admits a nontrivial (by quasianalyticity) complex Dulac expansion $\widehat{f}=\zeta+\beta+\exp(-\alpha_1\zeta)P_1(\zeta)+\mathrm{h.o.t.}$, where $\alpha_1 >0$ and $P_1\in \mathbb{C}\left[\zeta\right] $.
Hence $f$ satisfies the hypotheses of Theorem A in Section \ref{sec:analytic linearization dewsnap}. Therefore, by Theorem~A, $f$ admits an analytic linearization on the invariant domain $({\mathcal{R}_{C}})^f_{R}$, for $R>0$ sufficiently large. 

Note that here $({\mathcal{R}_{C}})^f_{R}=({\mathcal{R}_{C}})_{R}$, for $R>0$ sufficiently large. Indeed, by Example (3), it follows that if $\mathcal{R}_C$ is a standard quadratic domain, then for a sufficiently large $R>0$, $(\mathcal{R}_{C})_{R}$ is a domain of type $(\beta , \varepsilon , k)$, $\varepsilon >0$ and $k\in \mathbb{N}_{\geq 1}$. Proposition \ref{prop:invariance of D} implies that $(\mathcal{R}_{C})_{R}$, for a sufficiently large $R>0$, is $f$-invariant. Therefore, it is easy to see that $({\mathcal{R}_{C}})_{R}^f=(\mathcal{R}_{C})_{R}$, for a sufficiently large $R>0$. Now let $D:=(\mathcal{R}_{C})_{R}$. By Definition~\ref{def:gsqd}, we call such a domain $D$ a \emph{germ} of the standard quadratic domain $\mathcal R_C$, since, as a germ, it is equal to $\mathcal R_C$.\\

By Theorem A, there exists an analytic linearization $\varphi$ of $f$ on $D$, which is unique in the class of tangent to the identity linearizations on $f$-invariant subdomains. On the other hand, on the formal side, we apply the formal linearization Lemma \ref{lem:formal dulac linearization} to $\widehat{f}$, to obtain the unique formal parabolic Dulac linearization $\widehat \varphi$. We say that $f$ is \emph{formally linearizable} (by a parabolic complex Dulac series). \\

In the sequel, we prove the lemmas which will be used in the proof of Theorem~B in Subsection~\ref{sec:proofB}, to show that the formal linearization $\widehat{\varphi}$ is the Dulac asymptotic expansion of the analytic linearization $\varphi$ on some standard quadratic subdomain of $D$ 

%
The following Lemma \ref{lem:partial linearizations} shows how
the partial sums of the formal linearization of a hyperbolic complex Dulac
map perform its \emph{approximate} analytic linearizations on its domain. Such partial linearizations satisfy the homological equations \eqref{eq:coho}.
\begin{lem}[Partial linearizations]
	\label{lem:partial linearizations} Let $f(\zeta)=\zeta+\beta +o(1)$, $\beta\in\mathbb{C}^+$, be a nontrivial hyperbolic complex Dulac map
 defined on a standard quadratic domain $\mathcal R_C$. Let the parabolic Dulac series
	\begin{equation}
		\widehat{\varphi}\left(\zeta \right)=\zeta +\sum_{i=1}^{\infty}\mathrm{e}^{-\beta_{i}\zeta }Q_{i}\left(\zeta \right)\label{eq:serij},
	\end{equation}
	where $Q_{i}\in\mathbb{C}\left[\zeta\right]$ and  $(\beta_i)_i$ is a strictly increasing sequence of positive real numbers tending to $+\infty$, be its formal linearization from Lemma~\ref{lem:formal dulac linearization}. Here, if $\widehat\varphi$ is a finite sum, that is, if there exists $i_0\in\mathbb N$ such that $Q_i=0$ for $i> i_0$, we take any strictly increasing sequence $(\beta_i)_{i>i_0}$ such that $\beta_i>\beta_{i_0}$ and $\beta_i\to+\infty$.
	Let
	\begin{align}
		\widehat\varphi_{0} &:=\zeta, \nonumber \\
		\widehat\varphi_{n} &:=\zeta +\sum_{i=1}^{n}\mathrm{e}^{-\beta_{i}\zeta }Q_{i}(\zeta ),\ n\in\mathbb{N}_{\geq 1}, \label{eq:parsums}
	\end{align}
be	the \emph{partial sums} of $\widehat{\varphi}$, and $\varphi_n$ be the analytic germs on $\mathbb C^+$ defined by the finite sums $\widehat{\varphi}_n$, $n\in\mathbb N$. Then, for every $n\in\mathbb{N}$,
	there exists $\nu_{n}>0$, such that 
	\begin{equation}\label{eq:coho}
	(\varphi_{n}\circ f)(\zeta )-\varphi_{n}(\zeta ) =\beta +o\left(\mathrm{e}^{-(\beta_{n}+\nu_{n})\zeta }\right),
	\end{equation}
	uniformly on $\mathcal R_C$ as $\Re (\zeta ) \to +\infty $. Here, $\beta_{0}=0$. 
\end{lem}
Note that, if $\widehat\varphi$ is a finite sum, the sequence $(\varphi_n)_{n\in\mathbb{N}}$ eventually stabilizes.  

\begin{proof}
	Let 
	\[
	\widehat{f}\left(\zeta \right)=\zeta +\beta +\sum_{i=1}^{\infty}\mathrm{e}^{-\alpha_i \zeta }P_{i}\left(\zeta \right),\thinspace P_{i}\in\mathbb{C}\left[\zeta\right],\ i\in\mathbb N_{\geq 1}, \; 
	\]
	where $(\alpha_i)_i$ is a strictly increasing sequence of strictly positive real numbers tending to $+\infty$, be the (complex) Dulac expansion of $f$. Recall that this expansion of $f$ is uniform on a standard quadratic domain $\mathcal{R}_{C}$, as $\Re(\zeta) \to + \infty $. For $n\in\mathbb{N}$, let 
	\begin{align}
		\widehat{f}_{0} & := \zeta +\beta , \nonumber \\
		\widehat{f}_{n} &:=\zeta +\beta +\sum_{i\in\mathbb N_{\geq 1}:\,\alpha_i\leq\beta_{n}}\mathrm{e}^{-\alpha_i \zeta }P_{i}\left(\zeta \right), \quad n\in \mathbb{N}_{\geq 1}, \nonumber
	\end{align}
	be the \emph{partial sums} of $\widehat f$. 
	Furthermore, let $\widehat{g}_{n}:=\widehat{f}-\widehat{f}_{n}$, for $n\in \mathbb{N}$. The composition $\widehat{\varphi}\circ\widehat{f}$ can be computed as
	\begin{equation}
		\widehat{\varphi}\circ\widehat{f}=\widehat{\varphi}\left(\widehat{f}_{n}+\widehat{g}_{n}\right)=\widehat{\varphi}\circ\widehat{f}_{n}+\sum_{i\geq1}\frac{\widehat{\varphi}^{(i)}\circ\widehat{f}_{n}}{i!}\widehat{g}_{n}^{i} \label{Eq1},
	\end{equation}
	as the series in \eqref{Eq1} converges for the valuation topology.
Obviously, 
	\begin{equation}
		\widehat{\varphi}\circ\widehat{f}_{n}=\widehat{\varphi}_{n}\circ\widehat{f}_{n}+\left(\widehat{\varphi}-\widehat{\varphi}_{n}\right)\circ\widehat{f}_{n} , \quad n\in \mathbb{N}. \label{Eq2}
	\end{equation}
	Now, using \eqref{Eq1} and \eqref{Eq2} and the fact that $\widehat{\varphi }$ is the formal linearization of $\widehat{f}$,  we get: 
	\begin{align}
		0&=\widehat{\varphi}\circ\widehat{f}-\widehat{\varphi}-\beta\nonumber\\
		 &=\widehat{\varphi}_{n}\circ\widehat{f}_{n}-\widehat{\varphi}_{n}+\left(\widehat{\varphi}-\widehat{\varphi}_{n}\right)\circ\widehat{f}_{n}-\left(\widehat{\varphi}-\widehat{\varphi}_{n}\right)+\sum_{i\geq1}\frac{\widehat{\varphi}^{(i)}\circ\widehat{f}_{n}}{i!}\widehat{g}_{n}^{i} - \beta , \label{Eq3}
	\end{align}
	for $n\in \mathbb{N}$. For every $n\in\mathbb N$, it can easily be seen that there exists $\mu_n>0$, such that:
	\begin{align}
		\mathrm{ord}_{\mathrm{e}^{-\zeta}}\left(\left(\widehat{\varphi}-\widehat{\varphi}_{n}\right)\circ\widehat{f}_{n}-\left(\widehat{\varphi}-\widehat{\varphi}_{n}\right)+\sum_{i\geq1}\frac{\widehat{\varphi}^{(i)}\circ\widehat{f}_{n}}{i!}\widehat{g}_{n}^{i}\right) & >\beta _{n}+\mu _{n}.\label{Eq4}
	\end{align}
	From \eqref{Eq3} and \eqref{Eq4}, we obtain that
	\begin{equation}
		\mathrm{ord}_{\mathrm{e}^{-\zeta}}\left(\widehat{\varphi}_{n}\circ\widehat{f}_{n}-\widehat{\varphi}_{n}-\beta\right) >\beta_{n}+\mu _{n}, \quad n\in \mathbb{N}. \label{Eq5}
	\end{equation}
	As the sums in $\widehat{\varphi}_{n}$ and $\widehat{f}_{n}$ are
	finite, they define analytic germs $\varphi_{n}$ and $f_{n}$ on $\mathbb C^+$, in the $\zeta $-chart. This implies that 
	\begin{equation}
		\varphi_n\circ f_{n}-\varphi_{n}-\beta =o\left(\mathrm{e}^{-(\beta_{n}+\mu_{n})\zeta }\right), \ \Re(\zeta)\to+\infty \text{ on }\mathbb C^+, \label{Eq19}
	\end{equation}
	for $n\in \mathbb{N}$. Moreover, due to the fact that $f_n$ and $\varphi_n$, $n\in\mathbb N$, are \emph{finite} sums of power-exponential monomials, the convergence is \emph{uniform} if we restrict to the standard quadratic domain $\mathcal R_C$ (since the imaginary part is bounded by a power of the real part along this domain).
	
	Now put $ g_{n}(\zeta ):= f(\zeta )- f_{n}(\zeta )$, for $\zeta \in \mathcal{R}_{C}$ and $n\in \mathbb{N}$. It is obvious that $ g_{n}\sim\widehat{g}_{n}$ uniformly on $\mathcal{R}_{C}$ as $\Re (\zeta ) \to + \infty $. 
	By Taylor's Theorem, \eqref{Eq19}, and since $\varphi_n$ is a finite sum of monomials with uniform asymptotics on $\mathcal R_C$, it follows that  for every $n\in \mathbb{N}$ there exists some $\nu_n>0$ such that
	\begin{align}
		 \varphi _{n}\circ f-\varphi _{n}-\beta & =  \varphi _{n}( f_{n}+ g_{n})-\varphi _{n}-\beta\nonumber \\
		& = \varphi_{n}\circ  f_{n}-\varphi_{n}-\beta+ \sum _{i=1}^{+\infty }\frac{ \varphi_{n}^{(i)}( f_{n})}{i!} g_{n}^{i} \nonumber \\
		& = o(\mathrm{e}^{-(\beta_{n}+\nu _{n})\zeta }), \nonumber 
	\end{align}
	uniformly on the standard quadratic domain $\mathcal{R}_{C}$, as $\Re (\zeta )\to +\infty$.
\end{proof}

\medskip

Lemma~\ref{prop:homological equation} (below) shows how to solve and give an estimate of the solution
of a particular homological equation, that resembles the Abel's equation. The idea of the proof is taken from \cite{Loray_series_divergentes}. This will be used in the proof
of Theorem B (in Subsection~\ref{sec:proofB}) to control the growth of the differences between
the analytic linearization of a hyperbolic complex Dulac map, given by Theorem
A, and its partial linearizations by truncated complex Dulac sums given by \eqref{eq:parsums} in 
Lemma \ref{lem:partial linearizations}. These differences themselves
solve particular homological equations. 

This estimates allow us to conclude, in the proof of Theorem~B in Subsection~\ref{sec:proofB}, that a hyperbolic complex Dulac germ admits a parabolic complex Dulac linearization on some invariant standard quadratic (sub)domain of its domain of definition.

\begin{lem}[Explicit analytic solutions to Abel-type homological equations]
	\label{prop:homological equation}Let $f$ be a hyperbolic complex Dulac
	germ defined on a standard quadratic domain $\mathcal{R}_{C}\subset\mathbb{C}^+$. Let
	$h$ be an analytic map on $\mathcal{R}_{C}$, such that $h(\zeta )=o\left( \mathrm{e}^{-\alpha \zeta }\right)$ for some $\alpha >0$, uniformly on $\mathcal{R}_{C}$ as $\Re (\zeta ) \to +\infty $.
	Then:
	\begin{itemize}
		\item[$1.$] \emph{(Existence of an analytic solution to a homological equation)} There exist $R>0$ such that $D:=(\mathcal{R}_{C})_{R}$ is an $f$-invariant subdomain
		$D\subseteq\mathcal{R}_{C}$, and an analytic
		solution $\psi$ of the homological equation 
		\begin{equation}
			(\psi\circ  f)(\zeta )-\psi(\zeta )= h(\zeta ) \label{eq:homo}
		\end{equation}
		on the subdomain $D$.
		\item[$2.$] \emph{(Estimate of the solution)} The following estimate holds:
		\begin{equation}
			\psi(\zeta )=O(\mathrm{e}^{-\alpha\zeta}), \label{eq:asimptot}
		\end{equation}
		uniformly on $D$ as $\Re (\zeta )\to + \infty $.
		\item[$3.$] \emph{(Uniqueness of the solution)} If $\psi_{1}$ is an analytic solution of equation \eqref{eq:homo} on an $f$-invariant subdomain $D_{1}\subseteq \mathcal{R}_{C}$, such that $\psi _{1}(\zeta )=o(1)$ uniformly on $D_{1}$ as $\Re (\zeta ) \to + \infty $, then $$\psi _{1}\equiv \psi \text{ on } (D_{1})_R=D\cap D_1.$$
	\end{itemize} 
\end{lem}

\begin{proof}\

	\emph{1. Existence of a solution}. We prove that the following series: 
	\begin{equation}
		\psi(\zeta ):=-\sum_{n=0}^{+\infty} h\left( f^{\circ n}(\zeta )\right)\label{eq:psii}
	\end{equation}
	converges uniformly on $D$ (in the $\zeta $-chart) to an analytic map
	$\psi$ which satisfies equation \eqref{eq:homo}.
	
	Since $ h(\zeta )=o\left(\mathrm{e}^{-\alpha \zeta }\right)$ uniformly on $\mathcal{R}_{C}$
	as $\Re(\zeta )\to +\infty$, there exists $R>0$ such that
	\begin{equation}
		| h(\zeta )|\leq |\mathrm{e}^{-\alpha \zeta }| = \frac{1}{\mathrm{e}^{\alpha \Re (\zeta )}}\leq \frac{1}{\mathrm{e}^{\alpha R}}, \label{eq:prvis}
	\end{equation}
	for $\zeta \in \mathcal{R}_{C}$, $\Re (\zeta )\geq R$. 
	
	Let $\varepsilon>0$ and $k\in\mathbb N$ be arbitrary. By the discussion at the beginning of Subsection \ref{SubSectionCohomol}, we take $R>0$ sufficiently large such that $(\mathcal{R}_{C})_{R}=(\mathcal{R}_{C})_{R}^f$, that is, such that the whole of $(\mathcal{R}_{C})_R$ is $f$-invariant. Now, put $D:=(\mathcal{R}_{C})_{R}$. 
	From \eqref{Equat1} it follows that
	\begin{align}
		\Re ( f^{\circ n}(\zeta )) & \geq \Re(\zeta)+n\rho^-_{\beta , \varepsilon ,k}(R)\geq R+n\rho^-_{\beta , \varepsilon ,k}(R), \ \mathrm{for}\ \zeta\in D,\label{eqBefore}
	\end{align}
	for $n\in \mathbb{N}$. Now, from \eqref{eq:prvis} and \eqref{eqBefore}, it follows that, for $n\in\mathbb N$ and $\zeta\in D$
	\begin{align}
		\left|  h( f^{\circ n}(\zeta ))\right| &  \leq \frac{1}{\mathrm{e}^{\alpha \Re (f^{\circ n}(\zeta ))}} \leq \frac{1}{\mathrm{e}^{\alpha R}}\cdot \left( \frac{1}{\mathrm{e}^{\alpha \cdot\rho^-_{\beta , \varepsilon ,k}(R)}}\right)^{n} . \nonumber
	\end{align}
	This implies that sum \eqref{eq:psii} converges uniformly on $D$. By Weierstrass' Theorem, it follows that $\psi$ defined by \eqref{eq:psii} is analytic on $D$. Now \eqref{eq:homo} follows easily: 
	\begin{align*}
		\psi\left(f(\zeta )\right) & =-\sum_{n=0}^{+\infty} h\big( f^{\circ(n+1)}(\zeta )\big) \\
		&=-\sum_{n=1}^{+\infty} h\left( f^{\circ n}(\zeta )\right) \\
		& =-\big(-\psi(\zeta )-  h(\zeta )\big)\\
		& =\psi(\zeta )+  h(\zeta ),\text{ for }\zeta \in D.
	\end{align*}
	
	\vspace{0.2cm}
	
	\emph{2. Asymptotics of the solution} ${\psi }$. From \eqref{eq:psii},\ \eqref{eq:prvis} and
	\eqref{eqBefore} it follows that
	\begin{align}
		\left|\psi(\zeta )\right| & \leq \sum_{n=0}^{\infty}\left| h\left( f^{\circ n}(\zeta )\right)\right|\leq \mathrm{e}^{-\alpha \Re (\zeta )} \cdot \frac{1}{1- \frac{1}{\mathrm{e}^{\alpha \cdot \rho^- _{\beta , \varepsilon ,k}(R)}} } , \nonumber 
	\end{align}
	for $\zeta \in D$. This implies that $\psi(\zeta )=O\left(\mathrm{e}^{-\alpha\zeta}\right)$ uniformly on $D$ as
	$\Re (\zeta )\to + \infty $. \\
	
	3. \emph{Uniqueness of the solution.} Suppose that there exists an analytic
	solution $\psi_{1}$ to the homological equation \eqref{eq:homo},
	defined on an ${f}$-invariant subdomain $D_{1}\subseteq \mathcal{R}_{C}$, such that ${\psi }_{1}(\zeta )=o(1)$ uniformly on $D_{1}$ as $\Re (\zeta )\to +\infty$. By \eqref{eqBefore}, note that $(D_{1})_{R}=D_{1}\cap D$ is a nonempty ${f}$-invariant subdomain of $D$. Let 
	$$\psi _{2}(\zeta ):=\psi (\zeta )-\psi _{1}(\zeta ),\ \zeta \in (D_{1})_{R}.$$ 
	Since both $\psi$ and $\psi_1$ satisfy equation \eqref{eq:homo} on $(D_1)_R$ and $\psi (\zeta )=o(1)$, $ \psi_{1}(\zeta)=o(1)$, we have that $\psi _{2}( f(\zeta )) = \psi _{2}(\zeta )$, for $\zeta \in (D_{1})_{R}$, and $\psi _{2}(\zeta )=o(1)$ uniformly on $(D_{1})_{R}$ as $\Re (\zeta )\to + \infty $. Therefore, 
	\begin{equation}
		\psi _{2}( f^{\circ n}(\zeta )) = \psi _{2}(\zeta ), \ \zeta\in (D_1)_R,\ n\in \mathbb{N}.
		\label{eq:eqpsideux}
	\end{equation}
	Note that $\psi _{2}(\zeta )=o(1)$, as $\Re(\zeta)\to +\infty$ uniformly on $(D_1)_R$. From \eqref{eqBefore} it follows that $\Re(f^{\circ n}(\zeta))\to +\infty$ as $n \to \infty$, for every $\zeta\in (D_1)_R$. Therefore, passing to the limit as $n\to\infty$ in \eqref{eq:eqpsideux}, we obtain that  $\psi _{2}\equiv 0$ on $(D_{1})_{R}$. Therefore, $\psi\equiv \psi _{1}$ on $(D_{1})_{R}$.
\end{proof}

\subsection{Proof of Theorem B: uniqueness and existence of a (complex) Dulac linearization}\label{sec:proofB}\

We now gather all the previous results to finish the proof of Theorem B.
\smallskip
 
\begin{proof}[Proof of Theorem B]
	Let $f(\zeta)=\zeta+\beta+o(1)$, for $\beta \in \mathbb{C}^+$, be a hyperbolic complex Dulac germ on a standard quadratic domain
	$\mathcal{R}_{C}$ and let $\widehat{f}$ be its complex Dulac expansion.\\ 
	
	If $\widehat{f}=\zeta+\beta$, then, by \emph{quasianalyticity}, it follows that $f(\zeta)=\zeta+\beta$ for $\zeta\in \mathcal{R}_{C}$, so that $f$ is already linearized.\\
	
	Now, suppose that both $f$ and $\widehat{f}$ is nontrivial. By Lemma~\ref{lem:formal dulac linearization}, there exists a unique formal linearization $\widehat{\varphi}$
	of $\widehat{f}$, which is a parabolic complex Dulac series. By Theorem A, there exists a parabolic
	analytic linearization $\varphi$ of $f$ on the $f$-invariant subdomain $D:=(\mathcal{R}_{C})_{R}$, given as the uniform limit on $D$ of the Koenigs sequence for $f$. Note that $D$ is a ($f$-invariant) germ of $\mathcal R_C$.
	
To prove that $\varphi$ is a complex Dulac germ, we prove that it admits $\widehat\varphi$ as its asymptotic expansion, uniformly on some standard quadratic subdomain $\mathcal R_{C'}$ of $D$, as $\Re (\zeta ) \to + \infty $. 
	
	Let $\varphi_{n},\ n\in\mathbb{N}$, be the partial sums of the formal
	linearization $\widehat{\varphi}$ defined by \eqref{eq:parsums},
	and let 
	\begin{equation}
		\psi_{n}(\zeta ):=\varphi (\zeta )-\varphi_{n}(\zeta ), \; \zeta \in D, \, n\in\mathbb{N}. \label{eq:psij}
	\end{equation}
	Note that, by Theorem A (iii) and \eqref{eq:parsums},
	for every $\delta>0$ such that $\beta_1-\delta>0$ it holds that 
	\begin{equation}
		\psi_{n}(\zeta)=\zeta +o(1)- \zeta -o(\mathrm{e}^{-(\beta_{1}-\delta)\zeta })=o(1), \label{eq:asit}
	\end{equation}
	uniformly on $D$ as $\Re (\zeta )\to + \infty $. Since
	$\varphi$ is an analytic linearization of $ f$
on $D$, by \eqref{eq:psij} the following holds:
	\begin{equation}
		\psi_{n}( f(\zeta ))-\psi_{n}(\zeta )=-\varphi_{n}( f(\zeta )) + \varphi_{n}(\zeta ) +\beta , \text{ for }\zeta \in D\text{ and }n\in\mathbb{N}.\label{eq:ova}
	\end{equation}
	By Lemma \ref{lem:partial linearizations}, for every $n\in\mathbb N$ there exists $\nu_{n}>0$
	such that $(\varphi_{n}\circ  f)(\zeta )- \varphi_{n}(\zeta )=\beta+o\left(\mathrm{e}^{-(\beta_{n}+\nu_{n})\zeta }\right)$,
	uniformly on $\mathcal{R}_{C}$ as $\Re (\zeta) \to + \infty$. Here, $\beta_{n}>0$
	($n\in\mathbb{N}$) are the exponents in the complex Dulac series $\widehat{\varphi}$,
	as in \eqref{eq:serij}. 
	Now applying Lemma~\ref{prop:homological equation}
	to \eqref{eq:ova}, for every $n\in\mathbb{N}$, the homological
	equation \eqref{eq:ova} admits a unique solution $\eta_{n}$ analytic
	on $D$, such that $\eta_{n}(\zeta )=O(\mathrm{e}^{-(\beta _{n}+\nu_n)\zeta })$ uniformly on $D$, as $\Re (\zeta )\to + \infty $.

	Since $\psi_n=o(1)$ by \eqref{eq:asit}, 
	it follows from Lemma~\ref{prop:homological equation}\,$(3)$ that $\psi_{n}\equiv\eta_{n}$ on $D$,
	$n\in\mathbb{N}$. 
	Therefore, 
	$$\psi _{n}(\zeta )=O(\mathrm{e}^{-(\beta _{n}+\nu_n)\zeta }),\ n\in \mathbb{N}.$$
	This implies, by \eqref{eq:psij}, that $\widehat{\varphi}$ is the asymptotic expansion of the linearization $\varphi$ on $D$. Recall that $D$ is a germ of the standard quadratic domain $\mathcal R_C$, by Definition~\ref{def:gsqd}. 
	
	By Remark~\ref{rem:gdomain}, there exists a standard quadratic subdomain $\mathcal R_{C'}$, where $C'>R,C$, that is contained in $D$. Therefore, $\widehat{\varphi }$ is the Dulac asymptotic expansion of $\varphi $ also on the standard quadratic subdomain $\mathcal{R}_{C'}$. Thus, $\varphi$ is a parabolic complex Dulac germ (its domain of definition contains standard quadratic domains).\\

	Finally, the uniqueness of the linearization $\varphi$ follows from Theorem A $(iv)$. The statement about real Dulac linearizations of real Dulac germs follows from Theorem A $(ii)$.
\end{proof}


\def\cprime{$'$}
\providecommand{\bysame}{\leavevmode\hbox to3em{\hrulefill}\thinspace}
\providecommand{\MR}{\relax\ifhmode\unskip\space\fi MR }
\providecommand{\MRhref}[2]{%
	\href{http://www.ams.org/mathscinet-getitem?mr=#1}{#2}
}
\providecommand{\href}[2]{#2}

\emph{Addresses:}\

$^{1}$: University of Split, Faculty of Science, Ru\dj era Bo\v skovi\' ca 33, 21000 Split, Croatia, email: dino.peran@pmfst.hr

$^{2}$: University of Zagreb, Faculty of Science, Department of Mathematics, Bijeni\v cka 30, 10000 Zagreb, Croatia, email: maja.resman@math.hr
 
$^{3}$: Institut de Math\' ematiques de Bourgogne (UMR 5584 CNRS), Universit\' e de Bourgogne, Facult\' e des Sciences Mirande, 9 avenue Alain Savary, BP 47870, 21078 Dijon Cedex, France, email: jean-philippe.rolin@u-bourgogne.fr

$^{4}$: Universit\' e de Paris and Sorbonne Universit\' e, CNRS, IMJ-PRG, F-75006 Paris, France, email: tamara.servi@imj-prg.fr
\end{document}